\documentclass{amsart}

\usepackage{lineno}
\usepackage{amssymb}
\usepackage{amsmath}
\usepackage{amsthm}
\usepackage{mathrsfs}
\usepackage{bbm}
\usepackage{cite}
\usepackage{color}
\usepackage{enumitem}
\usepackage{pgfplots}

\usepackage[colorlinks,hypertexnames=false]{hyperref}
\usepackage[msc-links]{amsrefs}

\usepackage{changes}

\usepackage{mathtools}
\usepackage{extarrows}
\usepackage{setspace}

\usepackage{graphicx}

\usepackage{imakeidx}
\makeindex

\hypersetup{
	colorlinks=true,
	linkcolor=black,
	filecolor=black,
	urlcolor=black,
	citecolor=black,
}

\newtheorem{thm}{Theorem}[section]

\newtheorem{lem}[thm]{Lemma}

\newtheorem{theorem}[thm]{Theorem}
\newtheorem{corollary}[thm]{Corollary}

\newtheorem{definition}[thm]{Definition}

\newtheorem{remark}[thm]{Remark}

\numberwithin{equation}{section}

\pgfplotsset{compat=1.18} 

\begin{document}

\title[SFR functions on arbitrary domains]{Slice Fueter-regular functions on arbitrary domains in octonions}
\author{Xinyuan Dou}
\email[Xinyuan Dou]{douxinyuan@ustc.edu.cn}
\address{Department of Mathematics, University of Science and Technology of China, Hefei 230026, China}
\author{Guangbin Ren}
\email[Guangbin Ren]{rengb@ustc.edu.cn}
\address{Department of Mathematics, University of Science and Technology of China, Hefei 230026, China}
\author{Zeping Zhu}
\email[Zeping Zhu]{zzp@mail.ustc.edu.cn}
\address{Department of Mathematics, Chongqing Normal University, Chongqing 401331, China}  
\author{Ting Yang*}
\email[Ting Yang (The corresponding author)]{tingy@mail.ustc.edu.cn}
\address{School of Mathematics and Statistics, Anqing Normal University, Anqing, 246133, China}
%\date{\today}
\thanks{This work was supported by Xiaomi Young Talents Program and by the NNSF of China (12401104). }


\begin{abstract}
This paper is concerned with a class of generalized slice Fueter-regular functions on arbitrary domains in $\mathbb{O}$ with local stem functions.
Some classical theorems such as the maximum modulus principle will be generalized to our setting.
Some new phenomena such as the conditional uniqueness of stem vectors will be discovered by means of new   technical tools, e.g., the CCL equivalence relation and the Bers-Vekua continuation.
And a natural connection  between the theory of slice Fueter-regular functions and that of Riemann domains will be revealed via the quotient space under the CCL equivalence relation.
\end{abstract}

\maketitle
\section{Introduction}
The study of slice Dirac-regular functions, also known as slice Fueter-regular functions, was originally initiated  by M. Jin, G. Ren and I. Sabadini \cite{Jin-2020}, and latter enriched by R. Ghiloni with a certain type of Bers-Vekua system which characterizes axially monogenic functions of degree zero \cite{Ghiloni-2021}.
In light of the earlier contributions, we intend to broaden the current framework and seek for new observations.

More explicitly, this paper focuses on a particular class of generalized slice Fueter-regular functions, whose stem functions exist locally (but not necessarily globally), defined on arbitrary domains in $\mathbb{O}$. Firstly, we notice that the real analyticity still holds for such functions.
Combining the real analyticity and the fact that the restriction of any such function on every quaternion slice of its definition domain is harmonic, we are able to give a generalized and enhanced version of maximum modulus principle, along with its corollary, as follows.
\begin{itemize}
\item[i.](strong maximum modulus principle) Let $f$ be a nonconstant slice Fueter-regular function on a domain $\Omega\subset\mathbb{O}$. Then the octonion norm of the restriction of $f$ on any non-empty quaternion slice of $\Omega$ has no local maximum.
\item[ii.](weak maximum modulus principle) Let $f$ be a nonconstant slice Fueter-regular function on a domain $\Omega\subset\mathbb{O}$. Then the octonion norm of $f$ has no local maximum.
\end{itemize}
We would like to mention that the weak maximum modulus principle provides a negative answer to the question about the original version of maximum modulus principle for slice Fueter-regular functions posed by R. Ghiloni in Sect. 2.6 of \cite{Ghiloni-2021}.
Secondly, by means of a specific class of curves named as circular liftings, we establish the CCL equivalence relation on an arbitrary domain in $\mathbb{O}$, and discover a connection between slice Fueter-regular functions and Bers-Vekua continuations, i.e., real-analytic continuations satisfying the Bers-Vekua system (for more details, one may refer to Definitions \ref{def-circular-lifting}, \ref{def-CCL-equivalence}, \ref{def-Bers-Vekua-continuation}, and Theorem \ref{thm-con-between-SFR-function-&-BV-continuation}). And this connection eventually leads to a new phenomenon about the conditional uniqueness of stem vectors as follows.
\begin{itemize}
\item[iii.](uniqueness of stem vectors on CCL equivalence classes)
If two points $x, x'$ in a domain $\Omega\subset\mathbb{O}$ satisfy the CCL equivalence relation, then for any slice Fueter-regular function $f$ on $\Omega$, the stem vectors of $f$ at $x, x'$ are identical.
\end{itemize}
Finally, we notice that the CCL equivalence relation on any given domain $\Omega$ has a variant on the disjoint union (denoted as $\hat{\Omega}$) of the complex slices of $\Omega$. And based on several observations on this variant, particularly the natural local homeomorphism $P$ from its corresponding quotient space to $\mathbb{C}$, we are able to discover a Riemann domain structure strongly connected to the slice Fueter-regular functions as follow.
\begin{itemize}
\item[iv.](CCL Riemann domain)
The quotient space of $\hat{\Omega}$ under the variant CCL equivalence relation, equipped with the natural local homeomorphism $P$, is a Riemann domain over $\mathbb{C}$.
\end{itemize}
Since a lot of new notations have not been introduced here yet, the main results above are stated in a non-formal manner.
They will be rewritten formally as Theorem \ref{thm-maximum-modulus-principle-for-slice-Fueter-regular-functions}, Corollary \ref{cor-maximum-modulus-principle-for-slice-Fueter-regular-functions}, Theorems \ref{thm-conditional-uniqueness-stem-vectors} and \ref{thm-CCL-Riemann-domain}.

This paper is organized as follows.
In Section $2$, we mainly introduce a certain type of generalized slice Fueter regular functions, and provide a concrete example to illustrate the potential absence of the classical stem functions in our setting.
In Section $3$, we verify the real analyticity of such functions, and give a stronger version of maximum modulus principle as an application.
In Section $4$, based on the real analyticity, we design two main tools, i.e.,  ``CCL equivalence" and ``Bers-Vekua continuation", to study the conditional uniqueness of stem vectors.
In Section $5$, we show that ``CCL equivalence" naturally induces a Riemann domain structure.

\section{Preliminaries}
\subsection{Octonions}
Octonions are a well-known kind of hypercomplex numbers, constituting a non-associative division algebra over the field of real numbers.
This algebra is usually represented by the symbol $\mathbb{O}$.
A minimal introduction to its structure will be given below, one may refer to, e.g., \cites{Gentili-2008,Schafer-1966} for more details.

The algebra $\mathbb{O}$ has a real linear basis $\{e_0, e_1, \cdots, e_7\}$ subject to the following conditions:
$e_0$ is the multiplicative identity, in short $e_0=1$; and
$$
e_le_m=\left\{\begin{array}{ll}
                -e_0,                           & \text{if } l=m, \\
                &\\
                \epsilon_{lmn}e_n,              & \text{if } l\neq m;
              \end{array}\right.
$$
where $l,m=1,2,\cdots,7$, and $\epsilon_{lmn}$ is a completely antisymmetric tensor with value $1$ when $(l,m,n)$ is an even permutation of tripes $(1,2,3)$, $(1,4,5)$, $(1,7,6)$, $(2,4,6)$, $(2,5,7)$, $(3,4,7)$ or $(3,6,5)$, with value $-1$ when $(l,m,n)$ is an odd permutation of the listed tripes.
Every octonion can be written as $x=x_0+\sum_{l=1}^{7}x_le_l$.
Its conjugate is given as $\overline{x}=x_0-\sum_{l=1}^{7}x_le_l$, and its norm as $\lvert x\rvert=\sqrt{x\overline{x}}=\sqrt{\sum_{l=0}^{7}x_l^2}$.
$x_0$ is the real part of $x$, denoted by $\Re(x)$; and $\sum_{l=1}^{7}x_le_l$ is the imaginary  part, denoted by $\Im(x)$.

An imaginary unit is an octonion with unit norm and zero real part. The sphere consisting of all imaginary units is often denoted by $\mathbb{S}$.
The symbol $\mathbb{C}_I$ represents the subalgebra $\mathbb{R}+\mathbb{R}I$ (isomorphic to the complex field) generated by any given imaginary unit $I\in\mathbb{S}$.
It is worth mentioning that every octonion with non-zero imaginary  part can be uniquely decomposed into the form as $x_0+rI$ with $x_0$ being its real part, $r$ the norm of its imaginary part, and $I\in\mathbb{S}$. This observation leads to the so-called book structure of octonions, i.e., $\mathbb{O}=\cup_{I\in\mathbb{S}}\mathbb{C}_I$.

The set of ordered pairs comprising two orthogonal imaginary units is denoted by $\mathcal{N}$. The symbol $\mathbb{H}_\mathbb{I}$ stands for the subalgebra (isomorphic to the algebra of quaternions) generated by any given pair $\mathbb{I}=(I,J)\in\mathcal{N}$.
The Cayley-Dickson construction says $\mathbb{O}=\mathbb{H}_\mathbb{I}+L\mathbb{H}_\mathbb{I}$, where $L$ is another imaginary unit orthogonal to $\mathbb{H}_\mathbb{I}$.

\subsection{Slice functions on arbitrary domains}
Throughout this paper, $\mathbb{O}$ is naturally equipped with the ordinary Euclidean topology, since every octonion $x$ can be identified with the vector $(x_0, x_1,\cdots, x_7)$ in $\mathbb{R}^8$.
Let $\Omega$ be an arbitrary (open) domain in $\mathbb{O}$.
For any $a\in \mathbb{R}$ and $b\in(0,+\infty)$, define $\mathbb{S}(\Omega,a,b):=\{I\in\mathbb{S}|a+bI\in\Omega\}$.
\begin{definition}\label{def-slice-function}
A function $f:\Omega\rightarrow\mathbb{O}$ is called a $($left$)$ slice function if for every connected component $\Lambda$ of any non-empty $\mathbb{S}(\Omega,a,b)$, there exists a pair $(u,v)\in\mathbb{O}^2$ such that
\begin{equation}\label{eq-slice-function}
 f(a+bI)= u+Iv\quad\text{for all }I\in\Lambda.
\end{equation}
\end{definition}

We notice that in the definition above for any fixed $\Lambda$, the pair $(u,v)$ is inevitably unique, since $\Lambda$ contains infinitely many imaginary units and $\mathbb{O}$ is a division algebra.
%Indeed, $a+b\Lambda$ is an open sub-manifold of the sphere $a+b\mathbb{S}$, and $I$ can be considered as a smooth function on $a+b\Lambda$ with expression as
%\begin{equation*}
%I=\frac{\Im(x)}{b}, \qquad x\in a+b\Lambda.
%\end{equation*}
%Assume $(u',v')$ is another pair satisfying \eqref{eq-slice-function}.
%We thus have $(u'-u)\equiv I(v-v')$ on $a+b\Lambda$. Then applying the tangent vector fields $L_{mn}$ on both sides leads to
%\begin{equation}\label{eq-slice-function-1}
% L_{mn}I(v-v')\equiv 0 \text{ on }a+b\Lambda,
%\end{equation}
%for all $(m,n)\in\Theta$. A direct calculation yields
%$$L_{mn}I=\frac{(x_me_n-x_ne_m)}{b}. $$
%Substituting this equality into \eqref{eq-slice-function-1}, we see
%\begin{equation*}
%(x_me_n-x_ne_m)(v-v')\equiv 0 \text{ on }a+b\Lambda
%\end{equation*}
%for all $(m,n)\in\Theta$, which implies that $v=v'$, and thus $u=u'$.
Moreover, any given $x\in\Omega\setminus\mathbb{R}$ can be uniquely represented as $a+bI$ with $a\in \mathbb{R}$, $b\in(0,+\infty)$ and $I\in$ a certain $\Lambda$.
So there is a unique corresponding pair $(u,v)$ for such $x$. And we call $(u,v)$ as the stem vector of $f$ at the point $x$, written as $(u_x,v_x)$ to emphasize the dependence on $x$.
For the special case that $x\in\Omega\cap\mathbb{R}$, we set $(u_x,v_x):=(f(x),0)$.

For any $I\in\mathbb{S}$, between $\mathbb{C}$ and $\mathbb{C}_I$ there is a natural isomorphism with the following expression:
$$
\tau_I(\alpha+\beta i)=\alpha+\beta I
$$
for $\alpha,\beta\in\mathbb{R}$.
By convention, for any number $z\in \mathbb{C}$, $z_I$ stands for $\tau_I(z)$; for any domain $\Omega\subset\mathbb{O}$, $\Omega_I$ stands for $\Omega\cap\mathbb{C}_I$.
We denote $\tau_I^{-1}\Omega_I$ by $\Omega^I$, and separate it into two parts as
$$\Omega^I_{+}:=\Omega^I\cap\{\alpha+\beta i|\beta\geq 0\}\quad\text{and}\quad\Omega^I_{-}:=\Omega^I\cap\{\alpha+\beta i|\beta < 0\}.$$
With these notations, we are now able to reintroduce the concept of stem function.
\begin{definition}\label{def-stem-function}
Let $f:\Omega\rightarrow\mathbb{O}$ be a slice function. For each $I\in\mathbb{S}$, the stem function of $f$ on $\Omega^I$ is defined as
%$$
%(u^I,v^I): z\longmapsto \left\{\begin{array}{ll}
 %                            (u_{z_I},+v_{z_I}), & \text{ for }z\in \Omega^I_{+},\\
%                               &   \\
%                             (u_{z_I},-v_{z_I}), & \text{ for }z\in \Omega^I_{-},
%                           \end{array} \right.
%$$
$$
(u^I(z),v^I(z)):=(u_{z_I},\pm v_{z_I})\quad \text{for } z\in \Omega^I_{\pm},
$$
where $(u_{z_I},v_{z_I})$ is the stem vector of $f$ at  the point $z_I$.
\end{definition}
As we want, $f(z_I)=u^I(z)+Iv^I(z)$ is valid for any $z\in \Omega^I$.
When its dependence with the imaginary unit $I$ needs to be emphasized, we shall call $(u^I,v^I)$ the $I$-slice stem function.
\begin{remark}\label{rem-symmetry-of-stem-function}
Although the stem functions are not necessarily an even-odd pair in contrast with the axially symmetric case, they always possess a specific type of symmetry:
\begin{equation}\label{eq-symmetry-of-stem-function}
u^I(z)=u^{-I}(\overline{z})\quad\text{and}\quad v^I(z)=-v^{-I}(\overline{z})\quad\text{for } z\in\Omega^I, I\in\mathbb{S}.
\end{equation}
Moreover, if $I$ and $-I$ are path connected in every $\mathbb{S}(\Omega,a,b)$ containing both $I$ and $-I$, then the $I$-slice stem function will essentially become an even-odd pair, i.e.,
$$
u^I(z)=u^{I}(\overline{z})\quad\text{and}\quad v^I(z)=-v^{I}(\overline{z})\quad\text{for } z\in\Omega^{-I}\cap\Omega^I.
$$
\end{remark}
Now we introduce a sub-family of slice functions.
\begin{definition}
Let $f:\Omega\rightarrow\mathbb{O}$ be a slice function. If $(u^{I_1},v^{I_1})\equiv(u^{I_2},v^{I_2})$ o55n $\Omega^{I_1}\cap\Omega^{I_2}$ for any $I_1,I_2\in \mathbb{S}$, then we call $f$ a complete slice function.
\end{definition}
Briefly speaking, ``complete slice" means the existence of a global stem function.
Moreover, this function has an explicit expression as shown below.
\begin{definition}\label{def-global-stem-function}
Let $f:\Omega\rightarrow\mathbb{O}$ be a complete slice function. Its global stem function $($with domain of definition  $\cup_{I\in\mathbb{S}}\Omega^I$$)$ is defined as
$$
(u_\Omega(z),v_\Omega(z)):=(u^I(z),v^I(z)) \quad\text{for } z\in \text{some } \Omega^I.
$$
\end{definition}
\begin{remark}\label{rem-exist-of-global-stem-function}
It is easy to see
\begin{equation}\label{eq-f-expression-global-stem-function}
f(z_I)=u_\Omega(z)+Iv_\Omega(z)\quad\text{for } z\in\Omega^I, I\in\mathbb{S}.
\end{equation}
In this case, thanks to \eqref{eq-symmetry-of-stem-function}, $(u_\Omega,v_\Omega)$ is certainly an even-odd pair, i.e.,
$$
u_\Omega(z)=u_\Omega(\overline{z})\quad\text{and}\quad v_\Omega(z)=-v_\Omega(\overline{z})\quad\text{for } z\in\cup_{I\in\mathbb{S}}\Omega^I.
$$
\end{remark}
A special type of domain is described as follows.
\begin{definition}
Let $\Omega$ be a domain in $\mathbb{O}$.
If $\mathbb{S}(\Omega,a,b)$ is connected for any $a\in \mathbb{R}$ and $b\in(0,+\infty)$, then we call $\Omega$ a circularly connected domain.
\end{definition}
\begin{remark}\label{rem-local-axially-connectness}
Obviously, every open ball in $\mathbb{O}$ is a circularly connected domain.
Thus we can say that every non-empty open set in $\mathbb{O}$ is locally circularly connected, since all open balls form a topological basis of $\mathbb{O}$.
\end{remark}
Let $\Omega$ be a circularly connected domain and $f:\Omega\rightarrow\mathbb{O}$ a slice function.
Then the stem vector (as the solution to \eqref{eq-slice-function}) is independent of $I\in \mathbb{S}(\Omega,a,b)$ for any fixed $a,b$, since $\mathbb{S}(\Omega,a,b)$ is connected.
Hence the $I$-slice stem function is essentially independent of $I$ as well, namely
\begin{equation}\label{eq-compatibility-of-stem-functions}
(u^{I_1},v^{I_1})\equiv(u^{I_2},v^{I_2}) \quad\text{on } \Omega^{I_1}\cap\Omega^{I_2}
\end{equation}
for any $I_1,I_2\in \mathbb{S}$, which indicates the following conclusion.
\begin{theorem}\label{thm-complete-slice-on-circularly-connected-domain}
Every slice function on a circularly connected domain is a complete slice function.
\end{theorem}
Combing the theorem above with Remark \ref{rem-local-axially-connectness} yields that every slice function is locally a complete slice function. This observation leads to the definition of a local counterpart of the concept ``global stem function".
\begin{definition}\label{def-local-stem-function}
Let $f:\Omega\rightarrow\mathbb{O}$ be a slice function and $B$ be an arbitrary open ball in $\Omega$. The local stem function of $f$ on $B$ is defined as the global stem function of $f\rvert_B$, and denoted by $(u_B,v_B)$.
\end{definition}
Here $f\rvert_B$ stands for the restriction of $f$ on $B$.
\begin{remark}\label{rem-exit-of-local-stem-function}
Moreover, by Remark \ref{rem-exist-of-global-stem-function} we see the local stem function $(u_B,v_B)$ is an even-odd pair on its definition domain $\cup_{I\in\mathbb{S}}B^I$, and
\begin{equation}\label{eq-local-expression-of-f-by-stem-function}
f(z_I)=u_B(z)+Iv_B(z)
\end{equation}
holds for all $z$ and $I$ with $z_I\in B$.
\end{remark}
To better understand the concepts and propositions related to slice functions on arbitrary domains, we intend to give a concrete example. We consider the domain $\Omega$ made up of the following three parts:
$$
\Omega_0=\{x\in \mathbb{O}|\lvert \Im(x)\rvert< 1\}
$$
and
$$
\Omega_{\pm}=\mathbb{R}\pm \{tI|t\in[1,+\infty), I\in \mathbb{S}\text{ with } \arg(I_0,I)<\pi/4\}
$$
where $I_0$ is a fixed imaginary unit, and $\arg(\cdot,\cdot)$ represents the included angle between two assigned octonions. Then we define a continuous slice function on $\Omega$ as
$$
f(x)=\left\{\begin{array}{ll}
              0, & \text{ if }x\in \Omega_0,\\
               &  \\
             \pm x( \lvert \Im(x)\rvert-1),& \text{ if }x\in \Omega_\pm.
            \end{array}\right.
$$
Indeed, one can easily verify that for any given $a\in\mathbb{R}$ and $b\in(0,1)$, we have $\mathbb{S}(\Omega,a,b)=\mathbb{S}$, and the only solution to \eqref{eq-slice-function} is $(u,v)=(0,0)$; while for $a\in\mathbb{R}$ and $b\in[1,+\infty)$, $\mathbb{S}(\Omega,a,b)$ always has two connected components
$$\Lambda_{\pm}:=\{ I\in\mathbb{S}| \arg(\pm I_0,I)<\pi/4\},$$
and the solution to \eqref{eq-slice-function} is
$$(u,v)=\pm(a(b-1), b(b-1))$$
under the restriction $I\in\Lambda_{\pm}$.
Thus all the stems vectors of $f$ are determined. Explicitly, we see
$$(u_x,v_x)=(0,0)\quad\text{for }x\in \Omega_0,$$
and
$$
(u_x,v_x)=\pm\left(\Re(x)( \lvert \Im(x)\rvert-1), \lvert \Im(x)\rvert( \lvert \Im(x)\rvert-1)\right)\quad\text{for }x\in \Omega_\pm.
$$
It follows immediately that if $\arg(I_0,I)<\pi/4$, then by Definition \ref{def-stem-function} we have $\Omega^I=\mathbb{C}$, and
\begin{equation}\label{eq-expression-on-lambda-plus}
(u^{I}(z),v^{I}(z))=
\left\{\begin{array}{ll}
(0,0),& \text{ if } |\beta| < 1,\\
&\\
(\alpha(\beta-1), \beta(\beta-1)), &  \text{ if } \beta \geq +1, \\
&\\
(\alpha(\beta+1), \beta(\beta+1)), &  \text{ if } \beta \leq -1.
\end{array}\right.
\end{equation}
where $\alpha$ and $\beta$ are the real and imaginary parts of $z\in\mathbb{C}$ respectively.
Similarly,  we have
\begin{equation}\label{eq-expression-on-lambda-minus}
 (u^{-I}(z),v^{-I}(z))=
\left\{\begin{array}{ll}
(0,0),& \text{ if } |\beta| < 1,\\
&\\
(-\alpha(\beta-1), -\beta(\beta-1)) , &  \text{ if } \beta \geq +1,\\
&\\
(-\alpha(\beta+1), -\beta(\beta+1)), &  \text{ if } \beta \leq -1.
\end{array}\right.
\end{equation}
Based on the expressions above, we can tell that the $I$-slice stem function $(u^I,v^I)$ in this example is not an even-odd pair, and $(u^{\pm I},v^{\pm I})$ do satisfy \eqref{eq-symmetry-of-stem-function}, which confirms what we claimed in Remark \ref{rem-symmetry-of-stem-function}.

Now we choose $B$ in Definition \ref{def-local-stem-function} to be the open ball centered around $2I_0$ with radius $1$, then stitching together all the $I$-slice stem functions $(u^I,v^I)$ on nonempty $B^I$ ($I\in\mathbb{S}$) should results in the local stem function $(u_B,v_B)$. At first sight it appears impossible, because $(u^I,v^I)$ and $(u^{-I},v^{-I})$ apparently fail to satisfy the vital compatibility condition \eqref{eq-compatibility-of-stem-functions} at every point $z=\alpha+i\beta$ with $|\beta|>1$.
This may make one wonder what is exactly going on for the compatibility of the $I$-slice stem functions $(u^I,v^I)$ on nonempty $B^I$ ($I\in\mathbb{S}$).
In fact, there are essentially three possible cases in this example.
Case 1: $I_1$ and $I_2$ are both contained in the component $\Lambda_{+}$, then $(u^{I_1},v^{I_1})$ and $(u^{I_2},v^{I_2})$ share the same expression \eqref{eq-expression-on-lambda-plus}, meaning they are compatible.
Case 2: $I_1$ and $I_2$ are contained in the component $\Lambda_{-}$, then $(u^{I_1},v^{I_1})$ and $(u^{I_2},v^{I_2})$ share the same expression \eqref{eq-expression-on-lambda-minus}, meaning they are compatible as well.
Case 3: $I_1$ and $I_2$ are contained by the components $\Lambda_{+}$ and $\Lambda_{-}$ separately, then $B^{I_1}$ is located above the real axis, while $B^{I_2}$ is below the real axis, implying $B^{I_1}\cap B^{I_2}=\emptyset$; thereby, the compatibility condition \eqref{eq-compatibility-of-stem-functions} trivially holds on $B^{I_1}\cap B^{I_2}$. In conclusion, all $(u^I,v^I)$ on nonempty $B^I$ ($I\in\mathbb{S}$) are compatible, which ensures the existence of the local stem function $(u_B,v_B)$.
Furthermore, the explicit expression of $(u_B,v_B)$ is as follows:
\begin{equation*}
(u_B(z),v_B(z))=
\left\{
\begin{array}{ll}
(\alpha(\beta-1), \beta(\beta-1)), &  \text{ if } \beta > +1, \\
&\\
(-\alpha(\beta+1), -\beta(\beta+1)), &  \text{ if } \beta < -1,
\end{array}
\right.
\end{equation*}
for $z=\alpha+i\beta\in\cup_{I\in\mathbb{S}}B^I$.
According with Remark \ref{rem-exit-of-local-stem-function}, the local stem function $(u_B,v_B)$ is indeed an even-odd pair, even though every $I$-slice stem function $(u^I,u^I)$ is not an even-odd pair on the whole plane $\mathbb{C}$ for $I\in\Lambda_+\cup\Lambda_-$.

\subsection{Criteria for differentiable slice functions}
There exist multiple criteria for slice functions on axially symmetric domains, e.g., Lemma 3.2 in \cite{Ghiloni-2014}, Lemmas 13 and 44 in \cite{Ghiloni-2021}.
Although most of them are no longer valid in the non-axially-symmetric case,  Lemma 13 in \cite{Ghiloni-2021}, known as the differential sliceness criterion, fits in perfectly with our setting.

The so-called spherical tangential differential operators are given as
$$
L_{mn}=x_m\frac{\partial}{\partial x_n}-x_n\frac{\partial}{\partial x_m}
$$
where $m,n\in\{1,2,\cdots,7\}$ and $m<n$.
For convenience, we denote the set of all $(m,n)$ pairs which appear above by $\Theta$.
And the octonionic spherical Dirac operator (see \cite{Ghiloni-2021}) is defined by
$$
\Gamma=-\sum_{(m,n)\in\Theta}e_m(e_nL_{mn}).
$$
Note that for any octonion $x$ with non-zero imaginary part, the real vector space generated by $\{L_{mn}|_x\}_{m,n}$ is exactly the tangent space at the point $x$ of the $6$-dimensional sphere $a+b\mathbb{S}$ with $a=\Re(x)$ and $b=\lvert \Im(x)\rvert$. Furthermore, the restriction of every $L_{mn}$ to the sphere $a+b\mathbb{S}$ is a smooth tangent vector field on this very sphere.

Let $\Omega$ be an arbitrary domain in $\mathbb{O}$. By convention, a function $f:\Omega\rightarrow\mathbb{O}$ that has continuous $n$-th derivatives on $\Omega$ w.r.t real variables $x_0,x_1,\cdots,x_7$ is referred to as a function of class $\mathcal{C}^n$.
The family of such functions is denoted by $\mathcal{C}^n(\Omega, \mathbb{O})$. The following lemma is a generalized version of Lemma 12 in \cite{Ghiloni-2021}.
\begin{lem}\label{lem-Gamma-v-relation}
Assume $f:\Omega\rightarrow\mathbb{O}$ is a slice function of class $\mathcal{C}^1$.
Then for any $x\in\Omega$, we have
$$
6\Im(x)v_x=\lvert \Im(x)\rvert\Gamma f(x).
$$
\end{lem}
\begin{proof} The proof of this lemma is quite similar to that of Lemma 12 in \cite{Ghiloni-2021}.
So some details are omitted. For convenience, we replace the symbol $x$ by $y$, and proceed to prove
$$
6\Im(y)v_y=\lvert \Im(y)\rvert\Gamma f(y)
$$
holds for any given $y\in\Omega$.
Without loss of generality, we may assume $\Im(y)\neq 0$.
Let $\Lambda$ be the connected component containing $\frac{\Im(y)}{\lvert \Im(y)\rvert}$  of $\mathbb{S}(\Omega,a,b)$, where $a=\Re(y)$ and $b=\lvert \Im(y)\rvert$.
Recall that the stem vector $(u_y,v_y)$ is the solution to \eqref{def-slice-function} under the restriction $I\in\Lambda$, that is to say,
\begin{equation*}
 f(a+bI)= u_y+Iv_y
\end{equation*}
holds for all $I\in\Lambda$.
We notice that $a+b\Lambda$ is an open sub-manifold of the sphere $a+b\mathbb{S}$, and $I$ can be considered as a smooth function on $a+b\Lambda$ with expression as
\begin{equation*}
I=\frac{\Im(x)}{b} \quad\text{for } x\in a+b\Lambda.
\end{equation*}
Moreover, $L_{mn} f(x)$ can be identified with $L_{mn}|_x(f|_{a+b\Lambda})$ for any $x\in a+b\Lambda$,
since $L_{mn}|_x$ is a tangent vector to the sub-manifold $a+b\Lambda$ at the point $x$.
It follows immediately that
\begin{equation}\label{eq-Gamma-on-slice-1}
L_{mn} f(x)=L_{mn}|_x(u_y+Iv_y)=(L_{mn}|_xI)v_y.
\end{equation}
A direct calculation yields
$$L_{mn}|_xI=\frac{x_me_n-x_ne_m}{b}. $$
Substituting this into \eqref{eq-Gamma-on-slice-1}, we thus obtain
$$
L_{mn} f(x)=\frac{x_me_n-x_ne_m}{b}v_y,
$$
indicating
\begin{equation*}
\Gamma f(x)=-\sum_{(m,n)\in\Theta}e_m\left(e_n \left(\frac{x_me_n-x_ne_m}{b}v_y\right)\right).
\end{equation*}
As shown in the proof of Lemma 12 in \cite{Ghiloni-2021}, the so-called Moufang identities and Artin's theorem yield that
$$e_m(e_n(e_ma))=(e_me_ne_m)a=e_n a$$
and
$$e_m(e_n(e_na))=e_m((e_ne_n)a)=-e_ma$$
hold for any $a\in\mathbb{O}$.
Therefore,
\begin{equation*}
\Gamma f(x)=\sum_{(m,n)\in\Theta}\frac{x_me_m+x_ne_n}{b}v_y=\frac{6\Im(x)}{b}v_y \quad\text{for } x\in a+b\Lambda.
\end{equation*}
Taking $x=y$ completes the proof.
\end{proof}
Let $\Omega$ be an arbitrary domain in $\mathbb{O}$.
Recall that $\mathbb{S}(\Omega,a,b)=\{I\in\mathbb{S}|a+bI\in\Omega\}$ for any $a\in \mathbb{R}$ and $b\in(0,+\infty)$.
Assume $f:\Omega\rightarrow\mathbb{O}$ is a slice function of class $\mathcal{C}^1$. Then based on the technical lemma above, we see
\begin{equation}\label{eq-stem-function-of-differentiable-funciton}
(u_x,v_x)=\left(f(x)-\frac{1}{6}\Gamma f(x),\frac{\lvert \Im(x)\rvert}{6\Im(x)}\Gamma f(x)\right)
\end{equation}
for any $x\in\Omega$ with nonzero imaginary part.
This observation leads to a criterion for differentiable slice functions.
\begin{theorem}\label{thm-criterion-for-differentiable-slice-functions}
Assume $f:\Omega\rightarrow\mathbb{O}$ is a function of class $\mathcal{C}^1$. Then the following statements are equivalent.
\begin{enumerate}
\item[i.] $f$ is a slice function.
\item[ii.] For any given $a\in \mathbb{R}$ and $b\in(0,+\infty)$, the restrictions of $f-\frac{1}{6}\Gamma f$ and  $\Im^{-1}\Gamma f$ on $a+b\Lambda$ are always constants, where $\Lambda$ runs over all connected components of $\mathbb{S}(\Omega,a,b)$.
\end{enumerate}
\end{theorem}
To avoid possible misunderstandings, we would like to mention that $\Im^{-1}$ refers to the pointwise reciprocal of $\Im$ throughout this paper.
\begin{proof}
The necessity is obviously valid, since $\lvert \Im(x)\rvert\equiv b$ on $a+b\Lambda$ and the left side of \eqref{eq-stem-function-of-differentiable-funciton}  are indeed two constants on $a+b\Lambda$ provided that $f$ is a slice function. It only remains for us to show the sufficiency. Suppose that the restrictions of $f-\frac{1}{6}\Gamma f$ and  $\Im^{-1}\Gamma f$ on $a+b\Lambda$ are two constants, denoted as $u$ and $v'$ respectively. Now set $v=\frac{1}{6}bv'$, and we see
$$
f(x)=(f(x)-\frac{1}{6}\Gamma f(x))+\frac{1}{6}\Gamma f(x)=u+\frac{\Im(x)}{\lvert \Im(x)\rvert}v
$$
for any $x\in a+b\Lambda$, or equivalently,
$f(a+bI)=u+Iv$
for any $I\in\Lambda$. Therefore, $f$ is a slice function according to Definition \ref{def-slice-function}.
\end{proof}

A generalized version of the differential sliceness criterion, i.e., Lemma 13 in \cite{Ghiloni-2021}  follows immediately from the previous theorem.
\begin{corollary}\label{cor-criterion-for-differentiable-slice-functions}
Assume $f:\Omega\rightarrow\mathbb{O}$ is a function of class $\mathcal{C}^2$. Then the following statements are equivalent.
\begin{enumerate}
\item[i.]$f$ is a slice function.
\item[ii.]For any $(m,n)\in\Theta$, $L_{mn}(f-\frac{1}{6}\Gamma f)=L_{mn}(\Im^{-1}\Gamma f)=0$ on $\Omega\setminus\mathbb{R}$.
\end{enumerate}
\end{corollary}
\begin{proof}
Recall that for any given $a\in\mathbb{R}$, $b\in(0,+\infty)$ and any $x\in a+b\mathbb{S}$ , the real vector space generated by $\{L_{mn}|_x\}_{m,n}$ is exactly the tangent space to the sphere $a+b\mathbb{S}$ at the point $x$; and $a+b\Lambda$ is an open (connected) sub-manifold of
$a+b\mathbb{S}$ for every component $\Lambda$ of $\mathbb{S}(\Omega,a,b)$.
Hence, under the assumption that $f$ is $\mathcal{C}^2$, the restrictions of $f-\frac{1}{6}\Gamma f$ and  $\Im^{-1}\Gamma f$ on $a+b\Lambda$ are constants iff
$L_{mn}(f-\frac{1}{6}\Gamma f)=L_{mn}(\Im^{-1}\Gamma f)=0$ on $a+b\Lambda$ for any $(m,n)\in\Theta$. In addition, the union of all $a+b\Lambda$ is exactly $\Omega\setminus\mathbb{R}$, because every $x\in\Omega\setminus\mathbb{R}$ can be represented as $a+bI$, where $a=\Re(x)$, $b=\Im(x)$ and $I=\frac{\Im(x)}{\lvert \Im(x)\rvert}\in$ a certain $\Lambda$. What we have seen here indicates
the second assertion in this corollary is equivalent with that of Theorem \ref{thm-criterion-for-differentiable-slice-functions}. This completes the proof.
\end{proof}

\subsection{Slice Fueter-regular functions on arbitrary domains}\label{subsection-slice-Fueter-function}
A special differential operator strongly related to slice Fueter (or Dirac) regular functions, called slice Fueter operator, has been introduced by R. Ghiloni in \cite{Ghiloni-2021} as follows.
\begin{equation}
\overline{\partial}_F=\frac{\partial}{\partial x_0}-\Im^{-1}E-\frac{1}{3}\Im^{-1}\Gamma,
\end{equation}
where $E$ is the Euler operator $\sum_{l=1}^{7}x_l\frac{\partial}{\partial x_l}$. 
With certain previous content, now we are in a position to reintroduce the concept of slice Fueter-regular function in a wider framework.
Here $\Omega$ stands for an arbitrary domain in $\mathbb{O}$ as above.
\begin{definition}
A slice function $f:\Omega\rightarrow\mathbb{O}$ of class $\mathcal{C}^1$ is called a $($left$)$ slice Fueter-regular function if $\overline{\partial}_Ff\equiv 0$ on $\Omega\setminus\mathbb{R}$.
\end{definition}
Now we intend to give a concrete example of a slice Fueter-regular function on a non-axially-symmetric domain which is not a complete slice function. Its existence shows that not every slice Fueter-regular function has a slice Fueter-regular extension on an axially symmetric domain, since all slice functions on axially symmetric domains are complete slice ones. The detailed procedure of its construction and verification is quite lengthy. So we only present an outline here for the sake of simplicity.

Step $1$: Prepare the stem functions. By convention, the principal argument of a complex number $z$, written as $\arg(z)$, is restricted within the interval $(-\pi,+\pi]$.
For convenience, we let $\alpha$ and $\beta$ stand for the real and imaginary parts of $z$ respectively.
Now we construct a pair of real analytic functions on $O^+:=\{z\in\mathbb{C}|(z-2i)\notin(-\infty,0]\text{ and }\beta>0\}$ as follows.
\begin{equation}\label{eq-example-slice-Fueter-regular-stem-function-1}
\begin{split}
u(z):=&\frac{(\beta-2)\cos\left(\frac{\arg(z-2i)}{2}\right)-\alpha\sin\left(\frac{\arg(z-2i)}{2}\right)}{\beta(\alpha^2+(\beta-2)^2)^\frac{3}{4}}
,\\
   v(z):=&\frac{\alpha\cos\left(\frac{\arg(z-2i)}{2}\right)+(\beta-2)\sin\left(\frac{\arg(z-2i)}{2}\right)}{\beta(\alpha^2+(\beta-2)^2)^\frac{3}{4}} \\
     & -\frac{2(\alpha^2+(\beta-2)^2)^\frac{1}{4}\sin\left(\frac{\arg(z-2i)}{2}\right)}{\beta^2}.
\end{split}
\end{equation}
And it is not difficult to verify that
\begin{equation}\label{eq-example-slice-Fueter-regular-stem-function-2}
\begin{split}
\tilde{u}(z):=&\left\{\begin{array}{ll}
                        \text{sgn}(\beta-2)u(z), & \text{when } \beta\neq 2, \\
                        &\\
                         \frac{1}{2}\lvert\alpha\rvert^{-\frac{1}{2}}, &  \text{otherwise},
                      \end{array}\right.\\
\tilde{v}(z):=&\left\{\begin{array}{ll}
                        \text{sgn}(\beta-2)v(z), & \text{when } \beta\neq 2, \\
                        &\\
                         -\frac{1}{2}\lvert\alpha\rvert^{\frac{1}{2}}, &  \text{otherwise},
                      \end{array}\right.
\end{split}
\end{equation}
are real analytic on another domain $O^-:=\{z\in\mathbb{C}|(z-2i)\notin[0,+\infty)\text{ and }\beta>0\}$ by observing the complex analytic branches of $\sqrt{z-2i}$ on $O^-$.

Step $2$: Give the local definitions and check their compatibility.
We pick two orthogonal imaginary units $I,J$; and construct a path from $-J$ to $J$ across $I$ as
$$
\phi(\theta)=I\cos\frac{\theta}{2}+J\sin\frac{\theta}{2}\quad\text{for }\theta\in[-\pi,\pi].
$$
For each domain
$$B_\theta:=\left\{x\in\mathbb{O}\left|\lvert x-2\phi(\theta)-\exp(\theta\phi(\theta))\rvert<\frac{1}{4}\right.\right\},$$
we define a corresponding real analytic function $f_\theta: B_\theta\rightarrow\mathbb{O}$ in the following way.
In the first case that $\theta\in[-\frac{5\pi}{6},\frac{5\pi}{6}]$,
\begin{equation}\label{eq-example-slice-Fueter-regular-local-1}
 f_\theta(x):=u(z_x)+\frac{\Im(x)}{\lvert \Im(x)\rvert}v(z_x).
\end{equation}
In the second case that $\theta\in[-\pi,-\frac{5\pi}{6})$,
\begin{equation}\label{eq-example-slice-Fueter-regular-local-2}
 f_\theta(x):=-\tilde{u}(z_x)-\frac{\Im(x)}{\lvert \Im(x)\rvert}\tilde{v}(z_x).
\end{equation}
In the final case that $\theta\in(\frac{5\pi}{6},\pi]$,
\begin{equation}\label{eq-example-slice-Fueter-regular-local-3}
 f_\theta(x):=\tilde{u}(z_x)+\frac{\Im(x)}{\lvert \Im(x)\rvert}\tilde{v}(z_x).
\end{equation}
Here $z_x$ denotes $\Re(x)+i\lvert \Im(x)\rvert$.
We claim that these functions are pairwise compatible, i.e., $f_{\theta_1}\equiv f_{\theta_2}$ on $B_{\theta_1}\cap B_{\theta_2}$ for any $\theta_1,\theta_2\in [-\pi,\pi]$. Without loss of generality, we may assume $\theta_1<\theta_2$ and $B_{\theta_1}\cap B_{\theta_2}\neq\emptyset$, which implies $0<\theta_2-\theta_1<\frac{\pi}{3}$. Firstly, if $\theta_1$ and $\theta_2$ both belong to one of the three intervals $[-\pi,-\frac{5\pi}{6})$, $[-\frac{5\pi}{6},\frac{5\pi}{6}]$ and $(\frac{5\pi}{6},\pi]$, then there is nothing to prove. Secondly, if $\theta_1\in[-\pi,-\frac{5\pi}{6})$ and $\theta_2\in[-\frac{5\pi}{6},-\frac{\pi}{2})$, then $|\Im(x)|<2$ for any $x\in B_{\theta_1}\cap B_{\theta_2}$. By \eqref{eq-example-slice-Fueter-regular-stem-function-2}, we thus have
\begin{equation*}
   u(z_x)=-\tilde{u}(z_x) \text{ and } v(z_x)=-\tilde{v}(z_x),
\end{equation*}
indicating $f_{\theta_1}(x)= f_{\theta_2}(x)$ for all $x\in B_{\theta_1}\cap B_{\theta_2}$ due to \eqref{eq-example-slice-Fueter-regular-local-1} and \eqref{eq-example-slice-Fueter-regular-local-2}.
Thirdly, if $\theta_1\in(\frac{\pi}{2},\frac{5\pi}{6}]$ and $\theta_2\in(\frac{5\pi}{6},\pi]$, then
$|\Im(x)|>2$ for any $x\in B_{\theta_1}\cap B_{\theta_2}$, which similarly leads to the conclusion that $f_{\theta_1}(x)= f_{\theta_2}(x)$ still holds for all $x\in B_{\theta_1}\cap B_{\theta_2}$. Therefore,
the functions $f_{\theta}: B_\theta\rightarrow\mathbb{O}$ ($\theta\in[-\pi,\pi]$) are indeed pairwise compatible as claimed. Subsequently, combining them results in a real analytic slice function on the domain $\Omega:=\cup_{\theta\in[-\pi,\pi]}B_\theta$. We denote this function by $f$.

Step $3$: Verify the slice Fueter-regularity.
It suffices to prove $\overline{\partial}_Ff_0\equiv 0$ on $B_{0}$, because in this example $\overline{\partial}_Ff$ is real analytic (w.r.t. $x_0,x_1,\cdots,x_7$) on the domain $\Omega$.
Differentiating both sides of \eqref{eq-example-slice-Fueter-regular-local-1} yields
\begin{equation}\label{eq-example-slice-Fueter-regular-local-1-D}
\begin{split}
 \frac{\partial}{\partial x_0}f_0(x)=& \frac{\partial u}{\partial \alpha}(z_x)+\frac{\Im(x)}{\lvert \Im(x)\rvert}\frac{\partial v}{\partial \alpha}(z_x),\\
 \frac{\partial}{\partial x_k}f_0(x)=&\frac{x_k}{\lvert \Im(x)\rvert}\left(\frac{\partial u}{\partial \beta}(z_x)+\frac{\Im(x)}{\lvert \Im(x)\rvert}\frac{\partial v}{\partial \beta}(z_x)\right)\\
 &+\frac{e_k\lvert \Im(x)\rvert^2-x_k\Im(x)}{\lvert \Im(x)\rvert^3}v(z_x),
\end{split}
\end{equation}
where $k=1,2,\cdots,7$. It follows immediately from the second equality above that
\begin{equation}\label{eq-example-Ef}
  \Im(x)^{-1}Ef_0(x)=-\frac{\Im(x)}{\lvert \Im(x)\rvert}\frac{\partial u}{\partial \beta}(z_x)+\frac{\partial v}{\partial \beta}(z_x).
\end{equation}
In addition, Lemma \ref{lem-Gamma-v-relation} says
$$
\Im(x)^{-1}\Gamma f_0(x)=6\lvert \Im(x)\rvert^{-1}v(z_x),
$$
which together with \eqref{eq-example-Ef} and the first equality in \eqref{eq-example-slice-Fueter-regular-local-1-D} leads to
\begin{equation}\label{eq-example-DFf}
\begin{split}
   \overline{\partial}_Ff_0(x)=& \frac{\partial u}{\partial \alpha}(z_x)-\frac{\partial v}{\partial \beta}(z_x)-2\frac{1}{\lvert \Im(x)\rvert}v(z_x)
    \\
     & +\frac{\Im(x)}{\lvert \Im(x)\rvert}\left(\frac{\partial v}{\partial \alpha}(z_x)+\frac{\partial u}{\partial \beta}(z_x)\right). \\
     &
\end{split}
\end{equation}
Based on the explicit expressions in \eqref{eq-example-slice-Fueter-regular-stem-function-1}, direct calculation gives
\begin{equation}\label{eq-example-slice-Fueter-regular-stem-function-D-1}
\begin{split}
\frac{\partial u}{\partial \alpha}(z)=&\frac{-\alpha(\beta-2)\cos\left(\frac{\arg(z-2i)}{2}\right)+\frac{1}{2}(\alpha^2-(\beta-2)^2)\sin\left(\frac{\arg(z-2i)}{2}\right)}{\beta(\alpha^2+(\beta-2)^2)^\frac{7}{4}}
,\\
\frac{\partial v}{\partial \alpha}(z)=&-\frac{\frac{1}{2}(\alpha^2-(\beta-2)^2)\cos\left(\frac{\arg(z-2i)}{2}\right)+\alpha(\beta-2)\sin\left(\frac{\arg(z-2i)}{2}\right)}{\beta(\alpha^2+(\beta-2)^2)^\frac{7}{4}} \\
     & +\frac{(\beta-2)\cos\left(\frac{\arg(z-2i)}{2}\right)-\alpha\sin\left(\frac{\arg(z-2i)}{2}\right)}{\beta^2(\alpha^2+(\beta-2)^2)^\frac{3}{4}},\\
\frac{\partial u}{\partial \beta}(z)=&+\frac{\frac{1}{2}(\alpha^2-(\beta-2)^2)\cos\left(\frac{\arg(z-2i)}{2}\right)+\alpha(\beta-2)\sin\left(\frac{\arg(z-2i)}{2}\right)}{\beta(\alpha^2+(\beta-2)^2)^\frac{7}{4}} \\
     & -\frac{(\beta-2)\cos\left(\frac{\arg(z-2i)}{2}\right)-\alpha\sin\left(\frac{\arg(z-2i)}{2}\right)}{\beta^2(\alpha^2+(\beta-2)^2)^\frac{3}{4}},\\
\frac{\partial v}{\partial \beta}(z)=&\frac{-\alpha(\beta-2)\cos\left(\frac{\arg(z-2i)}{2}\right)+\frac{1}{2}(\alpha^2-(\beta-2)^2)\sin\left(\frac{\arg(z-2i)}{2}\right)}{\beta(\alpha^2+(\beta-2)^2)^\frac{7}{4}} \\
     & -\frac{2(\beta-2)\sin\left(\frac{\arg(z-2i)}{2}\right)+2\alpha\cos\left(\frac{\arg(z-2i)}{2}\right)}{\beta^2(\alpha^2+(\beta-2)^2)^\frac{3}{4}}\\
     &+\frac{4(\alpha^2+(\beta-2)^2)^\frac{1}{4}\sin\left(\frac{\arg(z-2i)}{2}\right)}{\beta^3}.
\end{split}
\end{equation}
Taking $z=z_x$ and inserting all the equalities of  \eqref{eq-example-slice-Fueter-regular-stem-function-D-1} and the second one of  \eqref{eq-example-slice-Fueter-regular-stem-function-1} into \eqref{eq-example-DFf}, we obtain $\overline{\partial}_Ff_0(x)=0$. It completes the procedure. In the end, we would like to mention that
the stem vector of $f$ at the point $-1+2(-J)$ is $(-\frac{1}{2},\frac{1}{2})$, while that of $f$ at the point $-1+2J$ is $(\frac{1}{2},-\frac{1}{2})$, indicating $f$ is not a complete slice function; and this very example actually comes from applying Fueter's Theorem (see, e.g., \cite{Sudbery-1979}) to the multi-valued complex analytic function $\sqrt{z-2i}$.

\section{Real analyticity of slice Fueter-regular functions}
In this section we shall discuss some topics related to the real analyticity of slice Fueter-regular functions, and establish a
generalized version of maximum modulus principle as an application.
Before we begin, we would like to mention that $\Omega$ represents an arbitrary domain in $\mathbb{O}$ by default throughout this section unless further specifications are provided.
\subsection{Representation formula}
To achieve our goal, it is unavoidable for us to deal with the relation between $2$, $4$ and $8$-dimensional real analyticities. And the representation formula is required as an important technical tool.
Recall that the symbol $\mathbb{S}(\Omega,a,b)$ denotes the set $\{I\in\mathbb{S}|a+bI\in\Omega\}$ where $a\in\mathbb{R}$ and $b\in(0,+\infty)$.
\begin{lem}\label{lemma-representation-formula}
Assume $f:\Omega\rightarrow\mathbb{O}$ is a slice function, and $\Lambda$ is a connected component of some $\mathbb{S}(\Omega,a,b)$.
Then for any $I_1,I_2\in \Lambda$ with $I_1\neq I_2$, the solution $(u,v)$ to \eqref{eq-slice-function}  can be expressed as
\begin{equation*}
\left\{\begin{aligned}
u=& f(a+bI_1)-I_1\left((I_1-I_2)^{-1} (f(a+bI_1)-f(a+bI_2))\right), \\
v=&(I_1-I_2)^{-1} (f(a+bI_1)-f(a+bI_2)
\end{aligned}\right.
\end{equation*}
\end{lem}
\begin{proof}
According to Definition \ref{def-slice-function}, we have
\begin{equation}\label{eq-proof-lem-represention-formula-1}
     u+I_kv=f(a+bI_k)\quad\text{for }k=1,2.
\end{equation}
Subtracting the two equalities above yields
\begin{equation*}
(I_1-I_2)v=f(a+bI_1)-f(a+bI_2),
\end{equation*}
indicating
\begin{equation*}
 v=(I_1-I_2)^{-1} (f(a+bI_1)-f(a+bI_2)).
\end{equation*}
Substituting this equality back into \eqref{eq-proof-lem-represention-formula-1} with $k=1$, we get
\begin{equation*}
     u=f(a+bI_1)-I_1\left((I_1-I_2)^{-1} (f(a+bI_1)-f(a+bI_2)\right).
\end{equation*}
The proof is completed.
\end{proof}

Recall that for any complex number $z=\alpha+\beta i$ and any imaginary unit $I\in \mathbb{S}$, the symbol $z_I$ stands for the octonion $\alpha+\beta I$.
When $\Omega$ is a circularly connected domain, an expression of the global stem function $(u_\Omega,v_\Omega)$ follows immediately from Lemma \ref{lemma-representation-formula}.
\begin{theorem}\label{thm-expression-global-stem-function}
Assume $f:\Omega\rightarrow\mathbb{O}$ is a slice function with $\Omega$ being a circularly connected domain. Then for any $z\in\mathbb{C}$ and $I_k\in\mathbb{S}$ $(k=1,2)$ satisfying $z_{I_k}\in\Omega$ $(k=1,2)$ and $I_1\neq I_2$, the following formulas hold true:
\begin{equation}\label{eq-thm-expression-global-stem-function}
\left\{\begin{aligned}
u_\Omega(z)=& f(z_{I_1})-I_1\left((I_1-I_2)^{-1} (f(z_{I_1})-f(z_{I_2}))\right), \\
v_\Omega(z)=&(I_1-I_2)^{-1} (f(z_{I_1})-f(z_{I_2})).
\end{aligned}\right.
\end{equation}
\end{theorem}
\begin{proof}
We notice that the two stem vectors at $z_{I_k}$ $(k=1,2)$, written as $(u_{z_{I_k}},v_{z_{I_k}})$, are identical, and they can be expressed by means of $f(z_{I_k})$ as follows.

Case $1$: $z=\alpha+\beta i$ with $\beta=0$. Then we have $z_{I_1}=z_{I_2}=\alpha$, and by definition
\begin{equation}\label{eq-proof-thm-expression-global-stem-function-1}
\left\{\begin{aligned}
u_{z_{I_k}}=&f(\alpha),\\
v_{z_{I_k}}=&0.
\end{aligned}
\right.
\end{equation}

Case $2$: $z=\alpha+\beta i$ with $\beta>0$. Then Lemma \ref{lemma-representation-formula} with $a=\alpha$, $b=\beta$ yields
\begin{equation}\label{eq-proof-thm-expression-global-stem-function-2}
\left\{\begin{aligned}
u_{z_{I_k}}=&f(z_{I_1})-I_1\left((I_1-I_2)^{-1} (f(z_{I_1})-f(z_{I_2}))\right),\\
v_{z_{I_k}}=&(I_1-I_2)^{-1} (f(z_{I_1})-f(z_{I_2})).
\end{aligned}
\right.
\end{equation}

Case $3$: $z=\alpha+\beta i$ with $\beta>0$. Similarly, taking $a=\alpha$, $b=-\beta$ and replacing $I_k$ by $-I_k$ in Lemma \ref{lemma-representation-formula}, we get
\begin{equation}\label{eq-proof-thm-expression-global-stem-function-3}
\left\{\begin{aligned}
u_{z_{I_k}}=&f(z_{I_1})-I_1\left((I_1-I_2)^{-1} (f(z_{I_1})-f(z_{I_2}))\right),\\
v_{z_{I_k}}=&-(I_1-I_2)^{-1} (f(z_{I_1})-f(z_{I_2})).
\end{aligned}
\right.
\end{equation}

Substituting each of the equalities \eqref{eq-proof-thm-expression-global-stem-function-1}, \eqref{eq-proof-thm-expression-global-stem-function-2} and \eqref{eq-proof-thm-expression-global-stem-function-3} into Definition \ref{def-stem-function} in the corresponding case, we thus obtain an expression of the $I_k$-slice stem function $(u^{I_k},v^{I_k})$ as
\begin{equation*}
\left\{\begin{aligned}
u^{I_k}(z)=& f(z_{I_1})-I_1\left((I_1-I_2)^{-1} (f(z_{I_1})-f(z_{I_2}))\right), \\
v^{I_k}(z)=&(I_1-I_2)^{-1} (f(z_{I_1})-f(z_{I_2})).
\end{aligned}\right.
\end{equation*}
Therefore, \eqref{eq-thm-expression-global-stem-function} is valid due to Definition \ref{def-global-stem-function}.
\end{proof}
\begin{remark}\label{rem-thm-expression-global-stem-function}
If we replace the assumption of the above theorem with that $f:\Omega\rightarrow\mathbb{O}$ is a complete slice function, then \eqref{eq-thm-expression-global-stem-function} still holds. Moreover, it is obvious that $f$ vanishes everywhere on $\Omega$ iff $(u_\Omega,v_\Omega)$ vanishes everywhere on $\cup_{I\in\mathbb{S}}\Omega^I$ according to \eqref{eq-f-expression-global-stem-function} and \eqref{eq-thm-expression-global-stem-function}.
\end{remark}

As a consequence of Theorem \ref{thm-expression-global-stem-function}, we get an analogous result to the representation formula for slice functions (see, e.g., Proposition 6 in Sect. 3.3 of \cite{Ghiloni-2011}).
\begin{corollary}\label{cor-representation-formula}
Assume $f:\Omega\rightarrow\mathbb{O}$ is a slice function with $\Omega$ being a circularly connected domain. Then for any $z\in\mathbb{C}$ and $I_k\in\mathbb{S}$ $(k=1,2,3)$ satisfying $z_{I_k}\in\Omega$ $(k=1,2,3)$ and $I_1\neq I_2$, the following formula holds true:
\begin{equation*}
f(z_{I_3})=(I_3-I_2)\left((I_1-I_2)^{-1} f(z_{I_1})\right)-(I_3-I_1)\left((I_1-I_2)^{-1}f(z_{I_2})\right).
\end{equation*}
\end{corollary}
The proof of this corollary is omitted, because it is essentially the same with that of Proposition 6 in
Sect. 3.3 of \cite{Ghiloni-2011}.

\begin{remark}\label{remark-expression-global-stem-function-local-version}
Now we turn back to the general case where the definition domain $\Omega$ is not necessarily circularly connected. As we know, every domain is a union of open balls and every open ball is circularly connected. Hence for an arbitrary slice function $f:\Omega\rightarrow\mathbb{O}$, a local version of \eqref{eq-thm-expression-global-stem-function} certainly holds true. More precisely, given any open ball $B\subset\Omega$, we have
\begin{equation}\label{eq-remark-expression-global-stem-function-local-version}
\left\{\begin{aligned}
u_B(z)=& f(z_{I_1})-I_1\left((I_1-I_2)^{-1} (f(z_{I_1})-f(z_{I_2}))\right), \\
v_B(z)=&(I_1-I_2)^{-1} (f(z_{I_1})-f(z_{I_2})),
\end{aligned}\right.
\end{equation}
for any $z\in\mathbb{C}$ and $I_k\in\mathbb{S}$ $(k=1,2)$ with $z_{I_k}\in B$ and $I_1\neq I_2$. In conclusion, the values of the stem function at all points can be recovered by means of the values of $f$. We would like to mention that a similar technique has been applied by G. Gentili and  C. Stoppato to establish a local version of the representation formula for quaternionic slice regular functions (see Theorem 3.4 in \cite{Gentili-2021}).
\end{remark}
\subsection{Fueter-regularity on quaternion slices}
In this part of our paper, we intend to verify the Fueter-regularity on quaternion slices of slice Fueter-regular functions. Then the real analyticity on quaternion slices follows as an immediate consequence. In order to describe the slice Fueter-regularity on axially symmetric domains in purely slice terms, a certain type of Bers-Vekua system has been thoroughly investigated by R. Ghiloni in \cite{Ghiloni-2021} based on many earlier contributions on related topics. To fit in with our notations, it is rewritten in the following form:
\begin{equation}\label{eq-Vekua}
  \left\{\begin{aligned}
             \frac{\partial u}{\partial \alpha}(z)-\frac{\partial v}{\partial \beta}(z)&=2\frac{v(z)}{\beta}  \quad\text{for }z\in O\setminus\mathbb{R},  \\
             \frac{\partial u}{\partial \beta}(z)+\frac{\partial v}{\partial \alpha}(z)&=0 \quad\text{for }z\in O,
         \end{aligned}
  \right.
\end{equation}
where $O$ is an open set in $\mathbb{C}$ and the functions $u,v:O\rightarrow\mathbb{O}$ are both $\mathcal{C}^1$. In line with the previous section, $\alpha$ and $\beta$ represent the real and  imaginary parts of the complex number $z$ respectively. Hereinafter, we call this very system the Bers-Vekua system on $O$ despite the fact that there are many other types of Bers-Vekua systems (see, e.g., Sect. 6.2 of \cite{Gurlebeck-2016}). As shown in Theorem 3 of \cite{Ghiloni-2021}, the stem function of every slice Fueter function on an arbitrary axially symmetric domain satisfies \eqref{eq-Vekua}. Not surprisingly, this property remains valid in a more general framework.
\begin{theorem}\label{thm-equiv-slice-Fueter-and-Vekua}
Assume $f:\Omega\rightarrow\mathbb{O}$ is a slice function of class $\mathcal{C}^1$. Then the following statements are equivalent.
\begin{enumerate}
\item[i.] $f$ is a slice Fueter-regular function, i.e., $\overline{\partial}_Ff\equiv 0$ on $\Omega\setminus\mathbb{R}$.
\item[ii.] For any open ball $B\subset\Omega$, the corresponding local stem function $(u_B,v_B)$ of $f$ satisfies the Bers-Vekua system, i.e., \eqref{eq-Vekua} with $O=\cup_{I\in\mathbb{S}}B^I$.
\end{enumerate}
\end{theorem}
\begin{proof}
It suffices to prove that $\overline{\partial}_Ff\equiv 0$ on $B\setminus\mathbb{R}$ iff $(u_B,v_B)$ satisfies \eqref{eq-Vekua} with $O=\cup_{I\in\mathbb{S}}B^I$.
For convenience, we assign a complex number $$z_x:=\Re(x)+i\lvert \Im(x)\rvert$$ to each $x\in B\setminus\mathbb{R}$.
Then replacing $z, I$ with $z_x, \frac{\Im(x)}{\lvert \Im(x)\rvert}$  in \eqref{eq-local-expression-of-f-by-stem-function}, we obtain
\begin{equation}\label{eq-local-expression-of-f-by-stem-function-2}
f(x)=u_B(z_x)+\frac{\Im(x)}{\lvert \Im(x)\rvert}v_B(z_x).
\end{equation}

Firstly, we claim that $(u_B,v_B)$ is $\mathcal{C}^1$ on $\cup_{I\in\mathbb{S}}B^I$ (w.r.t. $\alpha$ and $\beta$).
Indeed, for any given $z'\in\mathbb{C}$ and $I'\in\mathbb{S}$ with $z'_{I'}\in B$, there exists $\delta>0$ such that $z_I\in B$ holds for $z\in\mathbb{C}$ and $I\in\mathbb{S}$ with $\lvert z-z'\rvert<\delta$ and $\lvert I-I'\rvert<\delta$ due to the openness of $B$; then by \eqref{eq-remark-expression-global-stem-function-local-version} with $I_k$ $(k=1,2)$ satisfying $\lvert I_k-I'\rvert<\delta$ and $I_1\neq I_2$, we get
\begin{equation*}
\left\{\begin{aligned}
u_B(z)=& f(z_{I_1})-I_1\left((I_1-I_2)^{-1} (f(z_{I_1})-f(z_{I_2}))\right), \\
v_B(z)=&(I_1-I_2)^{-1} (f(z_{I_1})-f(z_{I_2})),
\end{aligned}\right. \quad\text{ when }\lvert z-z'\rvert<\delta,
\end{equation*}
which indicates that $(u_B,v_B)$ is $\mathcal{C}^1$ in a neighborhood of $z'$ under the given assumption of this theorem. Hence, making $z'$ run over the entire set $\cup_{I\in\mathbb{S}}B^I$, we see that $(u_B,v_B)$ is $\mathcal{C}^1$ on $\cup_{I\in\mathbb{S}}B^I$.

Secondly, Applying the first-order differential operators $\frac{\partial}{\partial x_0}$ and $E$ to the both sides of \eqref{eq-local-expression-of-f-by-stem-function-2} leads to

\begin{equation*}
 \frac{\partial f}{\partial x_0}(x)=\frac{\partial u_B}{\partial \alpha}(z_x)+\frac{\Im(x)}{\lvert \Im(x)\rvert}\frac{\partial v_B}{\partial \alpha}(z_x),
\end{equation*}
\begin{equation*}
   Ef(x)=\lvert \Im(x)\rvert\frac{\partial u_B}{\partial \beta}(z_x)+\Im(x)\frac{\partial v_B}{\partial \beta}(z_x).
\end{equation*}
Additionally, Lemma \ref{lem-Gamma-v-relation} yields
$$
\Gamma f(x)=6\frac{\Im(x)}{\lvert \Im(x)\rvert}v_x=6\frac{\Im(x)}{\lvert \Im(x)\rvert}v_B(z_x).
$$
Combing the three equalities above, we thus have
\begin{equation*}
\begin{split}
  \overline{\partial}_Ff(x)=&\frac{\partial f}{\partial x_0}(x)-\Im(x)^{-1}Ef(x)-\frac{1}{3}\Im(x)^{-1}\Gamma f(x) \\
  =&\frac{\partial u_B}{\partial \alpha}(z_x)-\frac{\partial v_B}{\partial \beta}(z_x)-\frac{2}{\lvert \Im(x)\rvert}v_B(z_x)
    \\
     & +\frac{\Im(x)}{\lvert \Im(x)\rvert}\left(\frac{\partial u_B}{\partial \beta}(z_x)+\frac{\partial v_B}{\partial \alpha}(z_x)\right).
\end{split}
\end{equation*}

In conclusion, the given assumption ensure that $(u_B,v_B)$ is $\mathcal{C}^1$, and
\begin{equation*}
\begin{split}
  \overline{\partial}_Ff(x)=&\frac{\partial u_B}{\partial \alpha}(z_x)-\frac{\partial v_B}{\partial \beta}(z_x)-\frac{2}{\lvert \Im(x)\rvert}v_B(z_x)
    \\
     & +\frac{\Im(x)}{\lvert \Im(x)\rvert}\left(\frac{\partial u_B}{\partial \beta}(z_x)+\frac{\partial v_B}{\partial \alpha}(z_x)\right)
\end{split}\quad\text{for }x\in B\setminus\mathbb{R},
\end{equation*}
or equivalently,
\begin{equation}\label{eq-proof-thm-slice-Fueter-and-Vekua-1}
  \overline{\partial}_Ff(z_I)=\frac{\partial u_B}{\partial \alpha}(z)-\frac{\partial v_B}{\partial \beta}(z)-\frac{2}{\beta}v_B(z)+I\left(\frac{\partial u_B}{\partial \beta}(z)+\frac{\partial v_B}{\partial \alpha}(z)\right)
\end{equation}
holds for any $z\in\mathbb{C}\setminus\mathbb{R}$ and $I\in \mathbb{S}$ with $z_I\in B\setminus\mathbb{R}$ due to the underlying fact that $u_B$ is even and $v_B$ is odd w.r.t. the variable $\beta$ as mentioned in Remark \ref{rem-exit-of-local-stem-function}. Thanks to \eqref{eq-proof-thm-slice-Fueter-and-Vekua-1}, one can easily see that $\overline{\partial}_Ff:B\setminus\mathbb{R}\rightarrow\mathbb{O}$ is a complete slice function, and its global stem function has the following expression:
\begin{equation*}
\left(\frac{\partial u_B}{\partial \alpha}(z)-\frac{\partial v_B}{\partial \beta}(z)-\frac{2}{\beta}v_B(z), \frac{\partial u_B}{\partial \beta}(z)+\frac{\partial v_B}{\partial \alpha}(z)\right)\quad\text{for }z\in\cup_{I\in\mathbb{S}}B^I\setminus\mathbb{R}.
\end{equation*}
Furthermore, as mentioned in Remark \ref{rem-thm-expression-global-stem-function}, we know that a complete slice function vanishes everywhere if and only if its global stem function vanished everywhere, which leads to the conclusion that
$\overline{\partial}_Ff\equiv 0$ on $B\setminus\mathbb{R}$ iff $(u_B,v_B)$ satisfies \eqref{eq-Vekua} with $O=\cup_{I\in\mathbb{S}}B^I$. Now the proof is completed.
\end{proof}

Recall that the symbol $\mathcal{N}$ represents the set of ordered pairs comprising two orthogonal imaginary units, and $\mathbb{H}_\mathbb{I}$ the quaternion sub-algebra of $\mathbb{O}$ generated by any given pair $\mathbb{I}=(I,J)\in\mathcal{N}$. As is well known, every quaternion $q$ in $\mathbb{H}_\mathbb{I}$ can be uniquely expressed as $q_0+q_1I+q_2J+q_3IJ$ with $q_k\in\mathbb{R}$ $(k=0,1,2,3)$. And the associated (left) Cauchy-Fueter operator (see, e.g., p140 of \cite{Colombo-2004}) is given by
$$
\frac{\partial}{\partial q_0}+I\frac{\partial}{\partial q_1}+J\frac{\partial}{\partial q_2}+IJ\frac{\partial}{\partial q_3},
$$
which is also called the slice Dirac operator by M. Jin, G. Ren and I. Sabadini in view of the slice structure of $\mathbb{O}$ mentioned in Proposition 3.2 of \cite{Jin-2020}. Out of respect for the previous contribution, we denote this operator by the same symbol $D_{\mathbb{I}}$ as in \cite{Jin-2020}. In addition, given any domain $\Omega\subset\mathbb{O}$ and $\mathbb{I}\in\mathcal{N}$, we denote $\Omega\cap\mathbb{H}_\mathbb{I}$ by $\Omega_\mathbb{I}$; and given any function $f:\Omega\rightarrow\mathbb{O}$, we denote the restriction of $f$ on $\Omega_\mathbb{I}$ by $f_\mathbb{I}$.

Before proceeding further, we would like to emphasize that throughout this paper, $x_0,\cdots,x_7$ always represent the real components of the octonion variable $x$, $q_0,\cdots,q_3$ always represent those of the quaternion variable $q$, and $\alpha, \beta$ always represent those of the complex variable $z$. We think the notations should be summarized here, because we may need to handle different types of variables at the same time in the later part of this paper.

To verify the Fueter-regularity on quaternion slices, we need the following lemma.
\begin{lem}\label{lemma-Vekua-Fueter-quaternion-slices}
Assume $f:B\rightarrow\mathbb{O}$ is a slice function of class $\mathcal{C}^1$, where $B$ is an open ball in $\mathbb{O}$.
Then the following statements are equivalent.
\begin{enumerate}
\item[i.] Its global stem function $(u_B,v_B)$ satisfies \eqref{eq-Vekua} with $O=\cup_{I\in\mathbb{S}}B^I$.
\item[ii.] $D_{\mathbb{I}}f_\mathbb{I}\equiv 0$ on $B_\mathbb{I}$ for any $\mathbb{I}\in\mathcal{N}$.
\end{enumerate}
\end{lem}
\begin{proof}
Just like deriving \eqref{eq-local-expression-of-f-by-stem-function-2}, we assign a complex number $$z_q:=\Re(q)+i\lvert \Im(q)\rvert$$ to each $q\in B_\mathbb{I}\setminus\mathbb{R}$, where $\Re(q)=q_0$ and $\Im(q)=q_1I+q_2J+q_3IJ$; then
\eqref{eq-local-expression-of-f-by-stem-function} with $z=z_q$ and $I=\frac{\Im(q)}{\lvert \Im(q)\rvert}$ yields
\begin{equation}\label{eq-local-expression-of-f-by-stem-function-3}
f_\mathbb{I}(q)=u_B(z_q)+\frac{\Im(q)}{\lvert \Im(q)\rvert}v_B(z_q).
\end{equation}
Under the given assumption of this lemma, $f_\mathbb{I}$ is certainly $\mathcal{C}^1$ on $B_\mathbb{I}$ (w.r.t. the real variables $q_0,\cdots,q_3$) for any $\mathbb{I}\in\mathcal{N}$.
Moreover, $u_B$ and $v_B$ are both $\mathcal{C}^1$ on $\cup_{I\in\mathbb{S}}B^I$ (w.r.t. the real variables $\alpha$ and $\beta$) as mentioned in the proof of Theorem \ref{thm-equiv-slice-Fueter-and-Vekua}.
Differentiating the both sides of \eqref{eq-local-expression-of-f-by-stem-function-3}, we obtain
\begin{equation}\label{eq-proof-lemma-Vekua-Fueter-quaternion-slices-1}
\frac{\partial f_\mathbb{I}}{\partial q_0}(q)=\frac{\partial u_B}{\partial \alpha}(z_q)+\frac{\Im(q)}{\lvert \Im(q)\rvert}\frac{\partial v_B}{\partial \alpha}(z_q)
\end{equation}
and
\begin{equation}\label{eq-proof-lemma-Vekua-Fueter-quaternion-slices-2}
\begin{split}
   \frac{\partial f_\mathbb{I}}{\partial q_k}(q)= & \frac{q_k}{\lvert \Im(q)\rvert}\frac{\partial u_B}{\partial \beta}(z_q)+\frac{q_k\Im(q)}{\lvert \Im(q)\rvert^2}\frac{\partial v_B}{\partial \beta}(z_q) \\
     & +\frac{\epsilon_k\lvert \Im(q)\rvert^2-q_k\Im(q)}{\lvert \Im(q)\rvert^3} v_B(z_q)
\end{split} \quad k=1,2,3,
\end{equation}
where $\epsilon_1=I$, $\epsilon_2=J$ and $\epsilon_3=IJ$.
A proper linear combination of \eqref{eq-proof-lemma-Vekua-Fueter-quaternion-slices-1} and \eqref{eq-proof-lemma-Vekua-Fueter-quaternion-slices-2} leads to
\begin{equation}\label{eq-proof-lemma-Vekua-Fueter-quaternion-slices-3}
\begin{split}
   D_{\mathbb{I}}f_\mathbb{I}(q)=& \frac{\partial u_B}{\partial \alpha}(z_q)-\frac{\partial v_B}{\partial \beta}(z_q)-\frac{2}{\lvert \Im(q)\rvert}v_B(z_q)  \\
     & +\frac{\Im(q)}{\lvert \Im(q)\rvert}\left(\frac{\partial u_B}{\partial \beta}(z_q)+\frac{\partial v_B}{\partial \alpha}(z_q)\right)
\end{split} \quad\text{for }q\in B_\mathbb{I}\setminus\mathbb{R}.
\end{equation}
Now we claim that \eqref{eq-proof-lemma-Vekua-Fueter-quaternion-slices-3} can be transformed into the the following conclusion: for any $z\in \cup_{I\in\mathbb{S}}B^I \setminus\mathbb{R}$, there exits a pair $\mathbb{I}=(I,J)\in\mathcal{N}$ such that \begin{equation}\label{eq-proof-lemma-Vekua-Fueter-quaternion-slices-4}
D_{\mathbb{I}}f_\mathbb{I}(z_I)=\frac{\partial u_B}{\partial \alpha}(z)-\frac{\partial v_B}{\partial \beta}(z)-\frac{2}{\beta}v_B(z)+I\left(\frac{\partial u_B}{\partial \beta}(z)+\frac{\partial v_B}{\partial \alpha}(z)\right).
\end{equation}
This new equality is apparently an analogous result to \eqref{eq-proof-thm-slice-Fueter-and-Vekua-1}. Because of the dependence w.r.t. $\mathbb{I}$, it requires more subtle manipulation to verify \eqref{eq-proof-lemma-Vekua-Fueter-quaternion-slices-4}. The details are as follows:

Given any $z\in \cup_{I\in\mathbb{S}}B^I \setminus\mathbb{R}$, there is obviously an imaginary unit $I\in\mathbb{S}$ so that $z\in B^I\setminus\mathbb{R}$.
%Moreover, thanks to the openness of $B$, there exits a positive constant $\delta$ such that for any $I'\in\mathbb{S}$ with $\lvert I'-I\lvert<\delta$,
%$z_{I'}\in B$.
Coupling this $I$ with another imaginary unit $J$ satisfying $J\perp I$ gives rise to a pair $\mathbb{I}=(I,J)\in\mathcal{N}$.
It is easy to see $z_I\in B_\mathbb{I}\setminus\mathbb{R}$, which allows
us to take $q=z_I$ on the left side of \eqref{eq-proof-lemma-Vekua-Fueter-quaternion-slices-3}, then the values of $z_q$ and $\Im(q)$ on the right side of \eqref{eq-proof-lemma-Vekua-Fueter-quaternion-slices-3} are correspondingly determined as
\begin{equation*}
  z_q=\alpha+\lvert\beta\lvert i,\quad \Im(q)=\beta I,
\end{equation*}
where $\alpha$ and $\beta$ are the real and imaginary parts of $z$ respectively. Thus we obtain
\begin{equation*}
  D_{\mathbb{I}}f_\mathbb{I}(z_I)=\left\{
  \begin{aligned}
    \frac{\partial u_B}{\partial \alpha}(z)-\frac{\partial v_B}{\partial \beta}(z)-\frac{2}{\beta}v_B(z)+I\left(\frac{\partial u_B}{\partial \beta}(z)+\frac{\partial v_B}{\partial \alpha}(z)\right),& \quad\text{if }\beta>0, \\
    \frac{\partial u_B}{\partial \alpha}(\overline{z})-\frac{\partial v_B}{\partial \beta}(\overline{z})-\frac{2}{\beta}v_B(\overline{z})-I\left(\frac{\partial u_B}{\partial \beta}(\overline{z})+\frac{\partial v_B}{\partial \alpha}(\overline{z})\right), & \quad\text{if }\beta<0.
  \end{aligned}
  \right.
\end{equation*}
In addition, Remark \ref{rem-exist-of-global-stem-function} says that $u_B$ is even and $v_B$ is odd  w.r.t. $\beta$, indicating $\frac{\partial u_B}{\partial \alpha}-\frac{\partial v_B}{\partial \beta}-\frac{2v_B}{\beta}$ is even and $\frac{\partial u_B}{\partial \beta}+\frac{\partial v_B}{\partial \alpha}$ is odd w.r.t. $\beta$.
Consequently, the above equality yields \eqref{eq-proof-lemma-Vekua-Fueter-quaternion-slices-4}.

It follows immediately from \eqref{eq-proof-lemma-Vekua-Fueter-quaternion-slices-3} that if Statement $i$ is assumed to be valid, then $D_{\mathbb{I}}f_\mathbb{I}\equiv 0$ on $B_\mathbb{I}\setminus\mathbb{R}$, indicating $D_{\mathbb{I}}f_\mathbb{I}\equiv 0$ on $B_\mathbb{I}$ due to the fact that $D_{\mathbb{I}}f_\mathbb{I}$ is continuous on $B_\mathbb{I}$. Hence, Statement $i$ implies Statement $ii$. Conversely, we may assume Statement $ii$ is valid, then \eqref{eq-proof-lemma-Vekua-Fueter-quaternion-slices-4} indicates that the complete slice function on $B\setminus\mathbb{R}$ generated by the stem function with the following expression
\begin{equation*}
\left(\frac{\partial u_B}{\partial \alpha}(z)-\frac{\partial v_B}{\partial \beta}(z)-\frac{2}{\beta}v_B(z), \frac{\partial u_B}{\partial \beta}(z)+\frac{\partial v_B}{\partial \alpha}(z)\right)\quad\text{for }z\in\cup_{I\in\mathbb{S}}B^I\setminus\mathbb{R}.
\end{equation*}
vanishes everywhere on $B\setminus\mathbb{R}$. According to Remark \ref{rem-thm-expression-global-stem-function}, we thus see the stem function
\begin{equation*}
\left(\frac{\partial u_B}{\partial \alpha}-\frac{\partial v_B}{\partial \beta}-\frac{2v_B}{\beta}, \frac{\partial u_B}{\partial \beta}+\frac{\partial v_B}{\partial \alpha}\right)
\end{equation*}
vanishes everywhere on $\cup_{I\in\mathbb{S}}B^I\setminus\mathbb{R}$, which means $(u_B,v_B)$ satisfies \eqref{eq-Vekua} with $O=\cup_{I\in\mathbb{S}}B^I$. Hence, Statement $ii$ implies Statement $i$. In conclusion, under the given assumption of this lemma, Statements $i$ and $ii$ are equivalent.
\end{proof}

We are now in a position to prove the Fueter-regularity on quaternion slices of slice Fueter-regular functions.
\begin{theorem}\label{thm-Fueter-regularity-on-quaternion-slices}
Assume $f:\Omega\rightarrow\mathbb{O}$ is a slice function of class $\mathcal{C}^1$. Then the following statements are equivalent.
\begin{enumerate}
\item[i.] $f$ is a slice Fueter-regular function, i.e., $\overline{\partial}_Ff\equiv 0$ on $\Omega\setminus\mathbb{R}$.
\item[ii.] For any $\mathbb{I}\in\mathcal{N}$, $f_\mathbb{I}$ is $($left$)$ Fueter-regular, i.e., $D_{\mathbb{I}}f_\mathbb{I}\equiv 0$ on $\Omega_\mathbb{I}$.
\end{enumerate}
\end{theorem}
\begin{proof}
Obviously, for any given $\mathbb{I}\in\mathcal{N}$, the union of $B_\mathbb{I}$ is exactly $\Omega_\mathbb{I}$ as $B$ traverse all open balls in $\Omega$. Thus it suffices to prove that $\overline{\partial}_Ff\equiv 0$ on $B\setminus\mathbb{R}$ iff $D_{\mathbb{I}}f_\mathbb{I}\equiv 0$ on $B_\mathbb{I}$ for all $\mathbb{I}\in\mathcal{N}$.
In the proof of Theorem \ref{thm-equiv-slice-Fueter-and-Vekua}, we have showed that $\overline{\partial}_Ff\equiv 0$ on $B\setminus\mathbb{R}$ iff $(u_B,v_B)$ satisfies \eqref{eq-Vekua} with $O=\cup_{I\in\mathbb{S}}B^I$.
Furthermore, Lemma \ref{lemma-Vekua-Fueter-quaternion-slices} says that
$(u_B,v_B)$ satisfies \eqref{eq-Vekua} with $O=\cup_{I\in\mathbb{S}}B^I$ iff $D_{\mathbb{I}}f_\mathbb{I}\equiv 0$ on $B_\mathbb{I}$ for all $\mathbb{I}\in\mathcal{N}$.
Therefore, $\overline{\partial}_Ff\equiv 0$ on $B\setminus\mathbb{R}$ iff $D_{\mathbb{I}}f_\mathbb{I}\equiv 0$ on $B_\mathbb{I}$ for all $\mathbb{I}\in\mathcal{N}$. This completes the proof.
\end{proof}

It follows immediately from Theorem \ref{thm-Fueter-regularity-on-quaternion-slices} that slice Fueter-regular functions are real analytic on every quaternion slice.
\begin{corollary}\label{cor-real-analytic-slice-Fueter-function-on-quaternion-slices}
Assume $f:\Omega\rightarrow\mathbb{O}$ is a slice Fueter-regular function. Then for any $\mathbb{I}\in\mathcal{N}$, the corresponding restriction $f_\mathbb{I}:\Omega_\mathbb{I}\rightarrow\mathbb{O}$ is harmonic, and thus analytic $($w.r.t. the real variables $q_0,\cdots,q_3$$)$.
\end{corollary}
\begin{proof}
Thanks to Theorem \ref{thm-Fueter-regularity-on-quaternion-slices}, it is easy to see that the splitting property demonstrated in Lemma 4.6 and Remark 4.7 of \cite{Jin-2020} remains valid for the generalized slice Fueter-regular functions introduced in this paper.
Indeed, under the given assumption of this corollary, for any $\mathbb{I}=(I,J)\in\mathcal{N}$ and any $L\in\mathbb{S}$ with $L\perp I,J$, there exist a left Fueter-regular function $f:\Omega_\mathbb{I}\rightarrow\mathbb{H}_\mathbb{I}$ and a right Fueter-regular function $G_{\mathbb{I},2}:\Omega_\mathbb{I}\rightarrow\mathbb{H}_\mathbb{I}$ such that $f_\mathbb{I}\equiv G_{\mathbb{I},1}+G_{\mathbb{I},2}L$. As shown in the proof of Theorem 11 of \cite{Sudbery-1979}, every Fueter-regular function is harmonic, which implies that $G_{\mathbb{I},1}$ and $G_{\mathbb{I},2}$  both are harmonic. We thus know $f_\mathbb{I}$ is harmonic as well.
This completes the proof.
\end{proof}
\subsection{Real analyticity of stem functions}
Based on what have been known in the earlier part of this section, we can now derive that all slice Fueter-regular functions are real analytic, so are their stem functions. First, we verify the real analyticity of the stem functions.
\begin{theorem}\label{thm-real-analytic-of-stem-function}
Assume $f:\Omega\rightarrow\mathbb{O}$ is a slice Fueter-regular function and $B$ is an open ball in $\Omega
$.
Then the corresponding local stem function $(u_B,v_B):\cup_{I\in\mathbb{S}}B^I\rightarrow\mathbb{O}^2$ is analytic $($w.r.t. the real variables $\alpha$ and $\beta$$)$.
\end{theorem}
\begin{proof}
Recall that in the proof of Theorem \ref{thm-equiv-slice-Fueter-and-Vekua}, we have seen that
for any given $z'\in\cup_{I\in\mathbb{S}}B^I$, there exist a positive constant  $\delta>0$ and two different imaginary units $I_1, I_2$  such that
\begin{equation*}
\left\{\begin{aligned}
u_B(z)=& f(z_{I_1})-I_1\left((I_1-I_2)^{-1} (f(z_{I_1})-f(z_{I_2}))\right), \\
v_B(z)=&(I_1-I_2)^{-1} (f(z_{I_1})-f(z_{I_2})),
\end{aligned}\right. \quad\text{ when }\lvert z-z'\rvert<\delta.
\end{equation*}
Obviously, one can always find a pair $\mathbb{I}=(I,J)\in\mathcal{N}$ satisfying $\mathbb{C}_{I_k}\subset\mathbb{H}_\mathbb{I}$ $(k=1,2)$.
Thus the equalities above can be rewritten as
\begin{equation}\label{eq-thm-real-analytic-of-stem-function-2}
\left\{\begin{aligned}
u_B(z)=& f_\mathbb{I}(z_{I_1})-I_1\left((I_1-I_2)^{-1} (f_\mathbb{I}(z_{I_1})-f_\mathbb{I}(z_{I_2}))\right), \\
v_B(z)=&(I_1-I_2)^{-1} (f_\mathbb{I}(z_{I_1})-f_\mathbb{I}(z_{I_2})),
\end{aligned}\right. \quad\text{ when }\lvert z-z'\rvert<\delta.
\end{equation}
We notice that $f_\mathbb{I}$ is real analytic as mentioned in Corollary \ref{cor-real-analytic-slice-Fueter-function-on-quaternion-slices}, and the mappings $z\to z_{I_k}$ $(k=1,2)$ are real linear, and thus real analytic.
As well known, the linear combinations and compositions of real analytic functions are still real analytic $($see, e.g., Propositions 2.2.2 and 2.2.8 in \cite{Krantz-2002}$)$. Therefore, due to \eqref{eq-thm-real-analytic-of-stem-function-2}, we see that $(u_B,v_B)$ is real analytic in a neighborhood of $z'$, which completes the proof.
\end{proof}

Theorem \ref{thm-real-analytic-of-stem-function} leads to the following corollary.
\begin{corollary}\label{cor-real-analytic-of-slice-Fueter-function}
Every slice Fueter-regular function is analytic $($w.r.t. the real variables $x_0,\cdots,x_7$$)$.
\end{corollary}
\begin{proof}
Assume $f:\Omega\rightarrow\mathbb{O}$ is a slice Fueter-regular function and $B$ is an arbitrary open ball in $\Omega$.
It suffices to prove $f$ is real analytic on $B$.
In the beginning of the proof of Theorem \ref{thm-equiv-slice-Fueter-and-Vekua}, we have showed that
for any $x\in B\setminus\mathbb{R}$, the following equality
\begin{equation}\label{eq-cor-real-analytic-of-slice-Fueter-function-1}
f(x)=u_B(z_x)+\frac{\Im(x)}{\lvert \Im(x)\rvert}v_B(z_x).
\end{equation}
holds true, where $z_x=\Re(x)+i\lvert\Im(x)\rvert$. On the right side of \eqref{eq-cor-real-analytic-of-slice-Fueter-function-1}, the terms $\frac{\Im(x)}{\lvert \Im(x)\rvert}$ and $z_x$ can be considered as real analytic function on $B\setminus\mathbb{R}$; moreover, $u_B$ and $v_B$ are both real analytic on $\cup_{I\in\mathbb{S}}B^I$. Thus \eqref{eq-cor-real-analytic-of-slice-Fueter-function-1} leads to the conclusion that $f$ is real analytic on $B\setminus\mathbb{R}$.

Now it remains to prove that $f$ admits a power series expansion around any point $a\in B\cap\mathbb{R}$. Firstly, we notice $u_B$ and $v_B$ are both real analytic around the real number $a$ due to Theorem \ref{thm-real-analytic-of-stem-function}.
Furthermore, as mentioned in Remark \ref{rem-exit-of-local-stem-function}, $(u_B,v_B)$ is an even-odd pair, that is, $u_B$ is even and $v_B$ is odd w.r.t. $\beta$, which implies that $$\left(\frac{\partial}{\partial\alpha}\right)^n\left(\frac{\partial}{\partial\beta}\right)^{2m+1}u_B(a)= \left(\frac{\partial}{\partial\alpha}\right)^n\left(\frac{\partial}{\partial\beta}\right)^{2m}v_B(a)=0$$
holds for all non-negative integers $n$ and $m$. Hence, the Taylor expansions of $u_B$ and $v_B$ around the point $a$ exist and take the following forms:
\begin{equation}\label{eq-cor-real-analytic-of-slice-Fueter-function-2}
\begin{split}
   u_B(z)&=\sum_{n,m=0}^{+\infty}C_{n,m}(\alpha-a)^n\beta^{2m}, \\
   v_B(z)&=\sum_{n,m=0}^{+\infty}C'_{n,m}(\alpha-a)^n\beta^{2m+1},
\end{split}
\end{equation}
where $C_{n,m}$ and $C'_{n,m}$ are octonion constants.
Replacing $\alpha$ and $\beta$ by $\Re(x)$ and $\lvert\Im(x)\rvert$ respectively on the right side of \eqref{eq-cor-real-analytic-of-slice-Fueter-function-2}, and $z$ by $z_x$ on the left side  correspondingly,
we thus have
\begin{equation*}
\begin{split}
   u_B(z_x)&=\sum_{n,m=0}^{+\infty}C_{n,m}\left(\Re(x)-a\right)^n\lvert\Im(x)\rvert^{2m}, \\
   v_B(z_x)&=\sum_{n,m=0}^{+\infty}C'_{n,m}\left(\Re(x)-a\right)^n\lvert\Im(x)\rvert^{2m+1}.
\end{split}
\end{equation*}
Substituting the equalities above into \eqref{eq-cor-real-analytic-of-slice-Fueter-function-1}, we finally obtain
\begin{equation*}
f(x)=\sum_{n,m=0}^{+\infty}\left(C_{n,m}+\Im(x)C'_{n,m}\right)\left(\Re(x)-a\right)^n\lvert\Im(x)\rvert^{2m},
\end{equation*}
which is essentially a Taylor expansion $($w.r.t. the real variables $x_0,\cdots,x_7$$)$ of $f$ around the point $a$, since the terms $\left(\Re(x)-a\right)^n$, $\Im(x)$ and $\lvert\Im(x)\rvert^{2m}$ are all homogeneous polynomials of $(x_0-a),x_1,\cdots,x_7$. This completes the proof.
\end{proof}

\subsection{Maximum modulus principle}
In the end of this section, we intend to establish  a new version of maximum modulus principle, to which the key is the following lemma.
\begin{lem}\label{lem-norm-subharmonic-on-quaternion-slice}
Assume $f:\Omega\rightarrow\mathbb{O}$ is a slice Fueter-regular function. Then $\lvert f_\mathbb{I}\rvert:\Omega_\mathbb{I}\rightarrow\mathbb{R}$ is subharmonic for any $\mathbb{I}\in\mathcal{N}$.
\end{lem}
\begin{proof}
A short approach is as follows:
We notice that the norm $\lvert\cdot\rvert:\mathbb{O}\rightarrow\mathbb{R}$ is apparently convex, and the function $f_\mathbb{I}:\Omega_\mathbb{I}\rightarrow\mathbb{O}$ is harmonic according to Corollary  \ref{cor-real-analytic-slice-Fueter-function-on-quaternion-slices}.
As well known, a composition of a convex function and a harmonic function is subharmonic (see, e.g., Theorem 3.4 of \cite{Ishihara-1979}), indicating that $\lvert f_\mathbb{I}\rvert:\Omega_\mathbb{I}\rightarrow\mathbb{R}$ is subharmonic.

In order to make this lemma more convincing, we would like to offer another more fundamental approach as follows.

Step $1$: Place $f_\mathbb{I}$ into a certain probability space.
For convenience, we denote the $3$-dimensional sphere in $\Omega_\mathbb{I}$ with center $q\in\Omega_\mathbb{I}$ and ratio $r>0$ by $\mathbb{S}_{\mathbb{I},q,r}$.
We may equip $\mathbb{O}$ with its Borel $\sigma$-algebra, and equip every $\mathbb{S}_{\mathbb{I},q,r}$ with its Borel $\sigma$-algebra (denoted as $\mathcal{B}_{\mathbb{I},q,r}$) as well.
In addition, we let $\mu_{\mathbb{I},q,r}$ stand for the uniform measure on the sphere $\mathbb{S}_{\mathbb{I},q,r}$, which is a probability measure invariant under the spherical rotations.
Since $f_\mathbb{I}:\Omega_\mathbb{I}\rightarrow\mathbb{O}$ is continuous, we know its restriction to any sphere $\mathbb{S}_{\mathbb{I},q,r}$ is measurable.

Step $2$: Apply Jensen's inequality to $f_\mathbb{I}$. Now we may identify $\mathbb{O}$ with $\mathbb{R}^8$. Then $f_\mathbb{I}:\mathbb{S}_{\mathbb{I},q,r}\rightarrow\mathbb{O}$ can be considered as a random vector (taking values in $\mathbb{R}^8$) defined on the probability space $(\mathbb{S}_{\mathbb{I},q,r},\mathcal{B}_{\mathbb{I},q,r},\mu_{\mathbb{I},q,r})$, and the norm $\lvert\cdot\rvert:\mathbb{O}\rightarrow\mathbb{R}$ can be considered as a convex real-valued function defined on $\mathbb{R}^8$.
Applying Jensen's inequality for the expectation of a convex real-valued function of
several real variables (see, e.g., Proposition 1.1 of \cite{Perlman-1974}), we thus obtain
\begin{equation}\label{eq-proof-lem-norm-subharmonic-on-quaternion-slice-1}
\left\lvert\int_{\mathbb{S}_{\mathbb{I},q,r}}f_\mathbb{I}  d\mu_{\mathbb{I},q,r}\right\rvert\leq\int_{\mathbb{S}_{\mathbb{I},q,r}}\lvert f_\mathbb{I}\rvert d\mu_{\mathbb{I},q,r}.
\end{equation}

Step $3$: Combine \eqref{eq-proof-lem-norm-subharmonic-on-quaternion-slice-1} with the mean value property.
As shown in Corollary  \ref{cor-real-analytic-slice-Fueter-function-on-quaternion-slices}, $f_\mathbb{I}:\Omega_\mathbb{I}\rightarrow\mathbb{O}$ is harmonic, which indicates that the left side of \eqref{eq-proof-lem-norm-subharmonic-on-quaternion-slice-1} equals $\lvert f_\mathbb{I}(q)\rvert$
according to the mean value property of harmonic functions. Hence, we see that the function $\lvert f_\mathbb{I}\rvert:\Omega_\mathbb{I}\rightarrow\mathbb{R}$ satisfies
\begin{equation*}
\lvert f_\mathbb{I}(q)\rvert\leq\int_{\mathbb{S}_{\mathbb{I},q,r}}\lvert f_\mathbb{I}\rvert d\mu_{\mathbb{I},q,r}.
\end{equation*}
for every sphere $\mathbb{S}_{\mathbb{I},q,r}$ in $\Omega_\mathbb{I}$, in other words, $\lvert f_\mathbb{I}\rvert:\Omega_\mathbb{I}\rightarrow\mathbb{R}$ is subharmonic.
This completes the proof.
\end{proof}

The real-analyticity of slice Fueter-regular functions  and the subharmonicity of their restrictions to quaternion slices lead to a strong version of maximum modulus principle.
\begin{theorem}[strong maximum modulus principle]\label{thm-maximum-modulus-principle-for-slice-Fueter-regular-functions}
If $f:\Omega\rightarrow\mathbb{O}$ is a nonconstant slice Fueter-regular function, then $\lvert f_\mathbb{I}\rvert:\Omega_\mathbb{I}\rightarrow\mathbb{R}$ has no local maximum for any  $\mathbb{I}\in\mathcal{N}$.
\end{theorem}
\begin{proof}
We may assume that $f:\Omega\rightarrow\mathbb{O}$ is a slice Fueter-regular function, and $\lvert f_\mathbb{I}\rvert:\Omega_\mathbb{I}\rightarrow\mathbb{R}$ has a local maximum at a point $x^*\in \Omega_\mathbb{I}$ for some  $\mathbb{I}\in\mathcal{N}$. Now we intend to prove $f:\Omega\rightarrow\mathbb{O}$ is a constant. This proof consists of three parts as follows.

Step $1$: Show that the restriction of $\lvert f_\mathbb{I}\rvert$ to some neighborhood is a constant via Lemma \ref{lem-norm-subharmonic-on-quaternion-slice}.  According to the definition of local maximum, we know that the maximum of the restriction of $\lvert f_\mathbb{I}\rvert$ to the $4$-dimensional open ball $\{q\in \mathbb{H}_\mathbb{I}: \lvert q-x^*\rvert\leq\delta\}$ is attained at the center $x^*$ when the radius $\delta$ is small enough.
For convenience, we denote the associated $8$-dimensional open ball $\{x\in \mathbb{O}: \lvert x-x^*\rvert\leq\delta\}$ by $B_\delta$.
One can easily see that the former ball is exactly $(B_\delta)_\mathbb{I}$, i.e., the intersection of $B_\delta$ and $\mathbb{H}_\mathbb{I}$.
Now we choose a sufficiently small $\delta$ to ensure that the following conditions hold:
\begin{enumerate}
\item[i.] The maximum of $\lvert f_\mathbb{I}\rvert: (B_\delta)_\mathbb{I}\rightarrow\mathbb{R}$ is attained at $x^*$.
\item[ii.] $B_\delta$ is included by the domain $\Omega$.
\end{enumerate}
Replacing the domain $\Omega$ in Lemma \ref{lem-norm-subharmonic-on-quaternion-slice} with this very ball $B_\delta$, we see that $\lvert f_\mathbb{I}\rvert: (B_\delta)_\mathbb{I}\rightarrow\mathbb{R}$ is subharmonic. As an immediate consequence, the maximum principle for subharmonic functions (see, e.g., Sect. 2.5 of \cite{Rosenblum-1994}) yields that $\lvert f_\mathbb{I}\rvert: (B_\delta)_\mathbb{I}\rightarrow\mathbb{R}$ equals some constant.

Step $2$: Enhance the conclusion of Step $1$ by removing the norm.
To achieve this goal, we introduce the auxiliary function $g:\Omega\to\mathbb{O}$ defined as
$$
g(x):=f(x)+f(x^*)\quad \text{for }x\in\Omega.
$$
Obviously, the triangle inequality for the norm indicates that
\begin{equation}\label{eq-thm-maximum-modulus-principle-for-slice-Fueter-regular-functions-1}
\lvert g_\mathbb{I}(x^*)\rvert=2\lvert f_\mathbb{I}(x^*)\rvert= \lvert f_\mathbb{I}(q)\rvert+\lvert f_\mathbb{I}(x^*)\rvert \geq \lvert g_\mathbb{I}(q)\rvert\quad \text{for }q\in(B_\delta)_\mathbb{I},
\end{equation}
which is to say that the maximum of $\lvert g_\mathbb{I}\rvert: (B_\delta)_\mathbb{I}\rightarrow\mathbb{R}$ is also attained at $x^*$. In addition, $g:\Omega\to\mathbb{O}$ is slice Fueter-regular as well. So a similar argument, the details of which we omit,  yields that $\lvert g_\mathbb{I}\rvert: (B_\delta)_\mathbb{I}\rightarrow\mathbb{R}$ equals some constant. This forces the equality to occur in the last part of \eqref{eq-thm-maximum-modulus-principle-for-slice-Fueter-regular-functions-1}, implying that $f_\mathbb{I}(q)$ and $f_\mathbb{I}(x^*)$, considered as vectors, share the same direction. Furthermore, the conclusion of Step $1$ says that $f_\mathbb{I}(q)$ and $f_\mathbb{I}(x^*)$ share the same norm.
Therefore, $f_\mathbb{I}(q)=f_\mathbb{I}(x^*)$ for any $q\in(B_\delta)_\mathbb{I}$, meaning  $f_\mathbb{I}: (B_\delta)_\mathbb{I}\rightarrow\mathbb{O}$ is a constant.

Step $3$: Apply the representation formula to $f:B_\delta\rightarrow\mathbb{O}$.
As in the proof of Theorem \ref{thm-equiv-slice-Fueter-and-Vekua}, we assign a complex number $z_x:=\Re(x)+\lvert\Im(x)\rvert i$ to each $x\in B_\delta$.
And we express the center $x^*$ as $\Re(x^*)+\lvert\Im(x^*)\rvert I^*$ where $I^*$ is an imaginary unit in the quaternion subalgebra $\mathbb{H}_{\mathbb{I}}$.
We claim that $(z_x)_{I^*}$, i.e., $\Re(x)+\lvert\Im(x)\rvert I^*$,  belongs to $(B_\delta)_\mathbb{I}$ for any $x\in B_\delta$.
Indeed, the definition of $B_\delta$ tells us $$\lvert\Re(x)-\Re(x^*)\rvert^2+\lvert\Im(x)-\Im(x^*)\rvert^2<\delta^2, $$
implying
$$\lvert\Re(x)-\Re(x^*)\rvert^2+\lvert\lvert\Im(x)\rvert-\lvert\Im(x^*)\rvert\rvert^2<\delta^2, $$
or equivalently,
$$\lvert(z_x)_{I^*}-x^*\rvert<\delta. $$
We thus know $(z_x)_{I^*}$ belongs to $B_\delta$, and thus belongs to $(B_\delta)_\mathbb{I}$ since $I^*$ is an imaginary unit in $\mathbb{H}_{\mathbb{I}}$.
It confirms what our claimed.

Now we fix the point $x$ temporarily, and consider the mapping $\tau_x:\mathbb{S}_{\mathbb{I}}\rightarrow\mathbb{H}_{\mathbb{I}}$ defined by
$$\tau_x(I):=(z_x)_I=\Re(x)+\lvert\Im(x)\rvert I\quad\text{for }I\in \mathbb{S}_{\mathbb{I}}, $$
where $\mathbb{S}_{\mathbb{I}}$ is the $3$-dimensional sphere consists of all imaginary unit in $\mathbb{H}_{\mathbb{I}}$.
$\tau_x$ is obviously continuous, which indicates that the preimage  $\tau_x^{-1}(B_\delta)_\mathbb{I}$ is an open subset of $\mathbb{S}_{\mathbb{I}}$.
Moreover, the fact that $\Re(x)+\lvert\Im(x)\rvert I^*$ belongs to $(B_\delta)_\mathbb{I}$ tells that $\tau_x^{-1}(B_\delta)_\mathbb{I}$ is non-empty.
These observations enable us to pick two different imaginary units $I_1$ and $I_2$ in $\tau_x^{-1}(B_\delta)_\mathbb{I}$. In addition, we take $I_3=\frac{\Im(x)}{\lvert\Im(x)\rvert}$ if $\Im(x)\neq 0$, and $I_3=$ an arbitrary imaginary unit if $\Im(x)= 0$.
Then we apply the representation formula in Corollary \ref{cor-representation-formula} with $\Omega$ substituted by $B_\delta$ and $z$ by $z_x$ to obtain
\begin{equation*}
f((z_x)_{I_3})=(I_3-I_2)\left((I_1-I_2)^{-1} f((z_x)_{I_1})\right)-(I_3-I_1)\left((I_1-I_2)^{-1}f((z_x)_{I_2})\right).
\end{equation*}
In our current setting, $(z_x)_{I_3}=x$, and $(z_x)_{I_k}$ $(k=1,2)$ belong to $(B_\delta)_\mathbb{I}$. Based on the conclusion of Step $2$, we have
$$f((z_x)_{I_k})=f_\mathbb{I}((z_x)_{I_k})=f_\mathbb{I}(x^*)=f(x^*)\quad\text{for } k=1,2. $$
Thus we see
\begin{equation*}
f(x)=(I_3-I_2)\left((I_1-I_2)^{-1} f(x^*)\right)-(I_3-I_1)\left((I_1-I_2)^{-1}f(x^*)\right)=f(x^*)
\end{equation*}
holds for all $x\in B_\delta$.

In conclusion, the restriction of $f: \Omega\rightarrow\mathbb{O}$ to the open ball $B_\delta$ equals some constant, which implies $f: \Omega\rightarrow\mathbb{O}$ globally equals some constant since it is a real analytic function as shown by Corollary \ref{cor-real-analytic-of-slice-Fueter-function} and $\Omega$ is a domain.
The proof is now completed.
\end{proof}

Note that if $\lvert f\rvert: \Omega\rightarrow\mathbb{R}$  has a local maximum, then $\lvert f_\mathbb{I}\rvert: \Omega_\mathbb{I}\rightarrow\mathbb{R}$  has the same local maximum for some $\mathbb{I}\in\mathcal{N}$.
So a weaker version of maximum modulus principle follows immediately.
\begin{corollary}[weak maximum modulus principle]\label{cor-maximum-modulus-principle-for-slice-Fueter-regular-functions}
If $f:\Omega\rightarrow\mathbb{O}$ is a nonconstant slice Fueter-regular function, then $\lvert f\rvert:\Omega\rightarrow\mathbb{R}$ has no local maximum.
\end{corollary}
\begin{remark}
The weak maximum modulus principle for slice Fueter-regular functions gives a negative answer to Question 38 posed by R. Ghiloni in \cite{Ghiloni-2021}.
\end{remark}

\section{Conditional uniqueness of stem vectors}
As demonstrated by the example at the end of Section \ref{subsection-slice-Fueter-function}, for a slice Fueter-regular function $f:\Omega\mapsto\mathbb{O}$, even if two distinct points $x_1, x_2\in\Omega$ have the same real part and their imaginary parts share the same norm, it is still possible that $f$ have different stem vectors at these two points.
A question naturally arises: Under what conditions does the uniqueness of stem vectors hold?
We notice that among all the correct answers to this question, one is quite trivial and relatively strong: According to Definition \ref{def-slice-function}, we can easily see that if the points $x_1, x_2$ mentioned above satisfy the additional condition that the imaginary units $\frac{\Im(x_1)}{\lvert\Im(x_1)\rvert}$ and $\frac{\Im(x_2)}{\lvert\Im(x_2)\rvert}$ are located in the same connected component of $\mathbb{S}(\Omega,a,b)$ where $a$ is the common real part of $x_k$ and $b$ is the common norm of the imaginary parts of $x_k$ $(k=1,2)$, then the stem vectors of $f$ at these two points are identical.
This answer is apparently not satisfactory, which gives us a strong motivation to discover a weaker condition under which the uniqueness of stem vectors holds.
And based on the authors' personal preference, this new condition is expected to have a geometric style.

\subsection{Octonionic circular liftings of pathes in the complex plane}
Recall that the symbol $\mathbb{S}$ stands for the sphere consisting of all imaginary octonion units, and $\tau_I$ ($I\in\mathbb{S}$) is a function defined on $\mathbb{C}$ with the following expression:
$$\tau_I(\alpha+\beta i)=\alpha+\beta I$$
where $\alpha$ and $\beta$ are arbitrary real numbers.
Hereinafter, a continuous mapping from the interval $[0,1]$ to a topological space $X$ is referred to as a path in $X$.
Now we intend to introduce a fundamental concept related to the uniqueness of stem vectors.
\begin{definition}\label{def-circular-lifting}
Let $\gamma$ and $\Theta$ be pathes in $\mathbb{C}$ and $\mathbb{S}$ respectively. Then the mapping: $$[0,1]\to\mathbb{O}, \quad t\mapsto\tau_{\Theta(t)}(\gamma(t)), $$
is called the (octonionic) circular lifting of $\gamma$ via $\Theta$, and denoted by $\gamma_\Theta$.
\end{definition}
In addition, we call $\gamma$ the base of $\gamma_\Theta$, and $\Theta$ the spherical coordinate of $\gamma_\Theta$. Obviously, every circular lifting is a path in $\mathbb{O}$. And conversely, every path in $\mathbb{O}$ which avoids the real axis is a circular lifting as stated in the following theorem.

\begin{theorem}
 Assume that $\tilde{\gamma}$ is a path in $\mathbb{O}\setminus\mathbb{R}$. Then there exist a path $\gamma$ in $\mathbb{C}$ and a path $\Theta$ in $\mathbb{S}$ such that $\tilde{\gamma}=\gamma_\Theta$.
\end{theorem}
\begin{proof}
Taking $\gamma(t):=\Re(\tilde{\gamma}(t))+\lvert\Im(\tilde{\gamma}(t))\rvert i$ and $\Theta(t):=\frac{\Im(\tilde{\gamma}(t))}{\lvert\Im(\tilde{\gamma}(t))\rvert}$, one can easily see that $\tilde{\gamma}=\gamma_\Theta$ holds true.
\end{proof}
Note that a path in $\mathbb{O}$ which intersects the real axis is not necessarily a circular lifting. For example, the polygonal path connecting the three vertices $I, 0, J$ with $I,J\in\mathbb{S}$ isn't a circular lifting unless $I=\pm J$. But every path in $\mathbb{O}$ can always be approximated uniformly by circular liftings as the next theorem implies.
\begin{theorem}\label{thm-approximation-of-circular-lifting}
Assume that $\tilde{\gamma}$ is a path in $\mathbb{O}$. Then for any $\delta>0$ there exists a path $\gamma$ in $\mathbb{C}$ and a path $\Theta$ in $\mathbb{S}$ such that $\tilde{\gamma}$ shares the same endpoints with $\gamma_\Theta$, and
$$\max_{[0,1]}\lvert\tilde{\gamma}-\gamma_\Theta\rvert <\delta.$$
\end{theorem}
\begin{proof}
It suffices to prove this theorem in the case that $\tilde{\gamma}$ is a polygonal path in $\mathbb{O}$.  We may assume its expression is as follows:
\begin{equation}\label{eq-thm-cl-2-1}
\tilde{\gamma}(t)=\frac{t-t_{k-1}}{t_{k}-t_{k-1}}x_k+\frac{t_k-t}{t_{k}-t_{k-1}}x_{k-1},\quad t_{k-1}\leq t\leq t_k\ (k=1,\cdots,n),
\end{equation}
where $0=t_0<t_1<\cdots<t_n=1$ and $x_0,x_1\cdots,x_n\in\mathbb{O}$.
This proof is divided into three steps.

Step $1$: Push the vertices except the endpoints away from the real axis.
Set $x'_k:=x_k+\frac{\delta}{4}I$, where $I$ is an fixed imaginary octonion unit, if $x_k\in\mathbb{R}$ and $k\neq 0,n$; otherwise, set $x'_k:=x_k$.
Replacing the vertices $x_{k-1},x_k$ by $x'_{k-1},x'_k$ in \eqref{eq-thm-cl-2-1} gives rise to a new polygonal path which we denote by $\tilde{\gamma}'$. It is obvious that $\tilde{\gamma}$ shares the same endpoints with $\tilde{\gamma}'$, and
\begin{equation}\label{eq-thm-cl-2-2}
\max_{[0,1]}\lvert\tilde{\gamma}-\tilde{\gamma}'\rvert \leq\frac{\delta}{4}.
\end{equation}

Step $2$: Push the crossover points of the line segments $(x'_{k-1},x'_k)$ and the real axis away from the real axis. Now define the third polygonal path $\tilde{\gamma}''$ as follows.
For $k=1,\cdots,n$, if there is a $t^*_k\in(t_{k-1},t_k)$ such that $\tilde{\gamma}'(t^*_k)\in\mathbb{R}$, then $\tilde{\gamma}''|_{[t_{k-1},t_k]}$ is given by
$$
\tilde{\gamma}''(t):=\left\{\
\begin{array}{ll}
\frac{t-t_{k-1}}{t^*_{k}-t_{k-1}}x^*_k+\frac{t^*_k-t}{t^*_{k}-t_{k-1}}x_{k-1},& t_{k-1}\leq t\leq t^*_k,\\
&\\
\frac{t-t^*_{k}}{t_{k}-t^*_{k}}x_k+\frac{t_k-t}{t_{k}-t^*_{k}}x^*_k,& t^*_{k}\leq t\leq t_k,
\end{array}
\right.
$$
where $x^*_k:=\tilde{\gamma}'(t^*_k)+\frac{\delta}{4}I_k$ with $I_k$ being an imaginary octonion unit $\neq\pm\frac{\Im(x'_k)}{\lvert\Im(x'_k)\rvert}$;  otherwise, $\tilde{\gamma}''|_{[t_{k-1},t_k]}:=\tilde{\gamma}'|_{[t_{k-1},t_k]}$.
Evidently, $\tilde{\gamma}'$ shares the same endpoints with $\tilde{\gamma}''$, and
\begin{equation}\label{eq-thm-cl-2-3}
\max_{[0,1]}\lvert\tilde{\gamma}'-\tilde{\gamma}''\rvert \leq\frac{\delta}{4}.
\end{equation}

Step $3$: Verify that the polygonal path $\tilde{\gamma}''$ is a circular lifting. Based on the construction of the polygonal path $\tilde{\gamma}''$, one can easily see that the only possible crossover points of $\tilde{\gamma}''$ and $\mathbb{R}$ are the endpoints $x_0,x_n$. This fact ensures the well-definedness of the following spherical coordinate (which is a path in $\mathbb{S}$):
$$
\Theta(t):=\left\{
\begin{array}{ll}
\frac{\Im(\tilde{\gamma}''(t))}{\lvert\Im(\tilde{\gamma}''(t))\rvert},& 0< t<1,\\
&\\
\lim\limits_{s\to 0^+}\frac{\Im(\tilde{\gamma}''(s))}{\lvert\Im(\tilde{\gamma}''(s))\rvert},& t=0,\\
&\\
\lim\limits_{s\to 1^-}\frac{\Im(\tilde{\gamma}''(s))}{\lvert\Im(\tilde{\gamma}''(s))\rvert},& t=1.
\end{array}
\right.
$$
Define the base (which is a path in $\mathbb{C}$) paired with this spherical coordinate as
$$
\gamma(t):=\Re(\tilde{\gamma}''(t))+\lvert\Im(\tilde{\gamma}''(t))\rvert i, \quad 0\leq t\leq 1.
$$
Then a simple calculation yields $\tilde{\gamma}''=\gamma_\Theta$ on $(0,1)$, indicating
$\tilde{\gamma}''=\gamma_\Theta$ on $[0,1]$ since both sides are continuous.
Moreover, the desired inequality $\max\limits_{[0,1]}\lvert\tilde{\gamma}-\gamma_\Theta\rvert <\delta$ is now a direct consequence of \eqref{eq-thm-cl-2-2} and \eqref{eq-thm-cl-2-3}.
It completes the proof.
\end{proof}

It follows immediately that the connectedness of open sets in $\mathbb{O}$ can be described via circular liftings.
\begin{theorem}\label{thm-connectedness-via-CL}
A nonempty open set $O\subset\mathbb{O}$ is connected if and only if any two points in $O$ can be connected by a circular lifting in $O$, i.e., for any $x_0,x_1\in O$ there exists a circular lifting $\gamma_\Theta$ such that $\gamma_\Theta(k)=x_k$ $(k=0,1)$ and $\gamma_\Theta([0,1])\subset O$.
\end{theorem}
\begin{proof}
Firstly, such $O$ is connected iff $O$ is path-connected, since it is locally path-connected (see, e.g., Theorem 25.5 in \cite{James-2000}). Secondly, Theorem \ref{thm-approximation-of-circular-lifting} implies that for any path in $O$ there exists a circular lifting in $O$ with the same endpoints. It is now obvious that the theorem holds.
\end{proof}

Based on its definition and the three theorems above, we can say that ``circular lifting" is a natural concept which fits in perfectly with the spatial structure and the topological structure of $\mathbb{O}$.

\subsection{CCL equivalence}
A certain type of pairs of circular liftings plays a vital role in characterizing the uniqueness of stem vectors.
\begin{definition}\label{def-CCL}
If two circular liftings have a common base and their spherical coordinates have a common initial point, then them are called coupled circular liftings.
\end{definition}
For example, the following two pathes in $\mathbb{O}$ can be considered as a pair of coupled circular liftings.
$\tilde{\gamma}_1(t):= t-tI$ and $\tilde{\gamma}_2(t):=t-t(I\cos t+J\sin t)$ with $I,J\in\mathbb{S}$ and $I\perp J$.
Indeed, taking $\gamma(t):=t-ti$, $\Theta_1(t):= I$ and $\Theta_2(t):=I\cos t+J\sin t$, one can easily see that $\tilde{\gamma}_k$ is the circular lifting of $\gamma$ via $\Theta_k$ $(k=1,2)$ and  the spherical coordinates $\Theta_1,\Theta_2$ share the same initial point $I$, which makes $\tilde{\gamma}_1,\tilde{\gamma}_2$ a pair of coupled circular liftings.

Now we will introduce a binary relation by means of coupled circular liftings.
\begin{definition}\label{def-CCL-connectedness}
Let $\Omega$ be a domain in $\mathbb{O}$. Two points $x_1,x_2\in\Omega$ is said to be CCL-connected in $\Omega$, denoted as $x_1 \overset{\Omega}{\sim} x_2$, if they are the terminal points of two coupled circular liftings in $\Omega$, i.e., there exist coupled circular liftings $\gamma_{\Theta_1}, \gamma_{\Theta_2}$ such that $\gamma_{\Theta_k}([0,1])\subset \Omega$ and $\gamma_{\Theta_k}(1)=x_k$ $(k=1,2)$.
\end{definition}
It is evident that $\overset{\Omega}{\sim}$, as a binary relation on $\Omega$, is reflexive and symmetric.
But this relation is not necessarily transitive. Indeed, whether $\overset{\Omega}{\sim}$ is transitive depends on the geometric features of $\Omega$. For example, if $\Omega$ is star-shaped with respect to some point in the real axis, then $\overset{\Omega}{\sim}$ is transitive.

This reflexive and symmetric relation naturally induces an equivalence relation as follows.
\begin{definition}\label{def-CCL-equivalence}
Let $\Omega$ be a domain in $\mathbb{O}$.
Two points $x,x'\in\Omega$ is said to be CCL-equivalent in $\Omega$, denoted as $x \overset{\Omega}{\simeq} x'$, if there exists a sequence $x_0,x_1,\cdots,x_n$ $(n\in\mathbb{N})$ such that $x_0=x$, $x_n=x'$ and $x_{k-1}\overset{\Omega}{\sim}x_k$ $(k=1,\cdots,n)$.
\end{definition}
Intuitively speaking, $\overset{\Omega}{\simeq}$ can be considered as some sort of completion of $\overset{\Omega}{\sim}$, because $\overset{\Omega}{\simeq}$ is the strongest one among all the equivalence  relations which are weaker than $\overset{\Omega}{\sim}$.

\subsection{Bers-Vekua continuation}
Inspired by ``analytic continuation" (see, e.g., Definition 7.1 in \cite{Forster-1981}) in complex analysis, we intend to introduce an analogous concept named as ``Bers-Vekua continuation", which is one of our main tools to investigate the uniqueness of stem vectors.

For convenience, some new notations are given as follows.
Let $O$ be a nonempty open subset of $\mathbb{C}$.
We denote the family of real analytic mappings from $O$ to $\mathbb{O}$ by $\mathcal{A}(O,\mathbb{O})$ (here $O$ and $\mathbb{O}$ are considered as real analytic manifolds),
and denote the family of all pairs $(u,v)$ with $u,v\in\mathcal{A}(O,\mathbb{O})$ that satisfy the Bers-Vekua system on $O$ by $\mathcal{BV}(O)$, or equivalently,
$$
\mathcal{BV}(O):=\{(u,v)\in\mathcal{A}(O,\mathbb{O})\times\mathcal{A}(O,\mathbb{O})|\ (u,v)\text{ satisfies  \eqref{eq-Vekua}} \}.
$$
Let $z$ be a complex number.
On the union $$\cup_{O\ni z}\mathcal{BV}(O)$$ where $O$ runs over all open neighborhoods of $z$, an equivalence relation is defined in the following way: Two elements $(u,v)$ and $(u',v')$ in this union are related, written as $(u,v)\overset{z}{\sim}(u',v')$, precisely if there is an open neighborhood $O$ of $z$ such that $u|_O=u'|_O$ and $v|_O=v'|_O$.
One can easily see that $\overset{z}{\sim}$ is an equivalence relation.
We denote the equivalence class of an element $(u,v)$ in this union by $[(u,v)]_z$, and the set of all the equivalence classes by $\mathcal{BV}_z$. In addition, we call $[(u,v)]_z$ the Bers-Vekua germ of $(u,v)$ at the point $z$, and $\mathcal{BV}_z$ the Bers-Vekua stalk at the point $z$.

Now  we are ready to give an explicit definition of Bers-Vekua continuation.
\begin{definition}\label{def-Bers-Vekua-continuation}
Let $\gamma$ be a path in $\mathbb{C}$. A Bers-Vekua continuation along the path $\gamma$ is a family of Bers-Vekua germs $\phi_t$ $(0\leq t\leq 1)$  satisfying the following condition:
For any $t^*\in[0,1]$, there exist a neighborhood $T\subset [0,1]$ of $t^*$, an open set $O\subset\mathbb{C}$ with $\gamma(T)\subset O$ and a pair $(u,v)\in \mathcal{BV}(O)$ such that
      $$\phi_t=[(u,v)]_{\gamma(t)}\quad\text{for every }t\in T.
      $$
\end{definition}

\begin{remark}\label{rem-Bers-Vekua-continuation}
The condition above implies that $\phi_t$ is a Bers-Vekua germ at the point $\gamma(t)$ for every $t\in[0,1]$. Furthermore, due to the compactness of the interval $[0,1]$, the condition above can be rewritten equivalently as:
There exist a partition $0=t_0<t_0<\cdots<t_n=1$ and open sets $O_k\subset \mathbb{O}$ with $\gamma([t_{k-1},t_k])\subset O_k$ and pairs $(u_k,v_k)\in\mathcal{BV}(O_k)$ $(k=1,\cdots,n)$ such that
$$\phi_t=[(u_k,v_k)]_{\gamma(t)}\quad\text{for every }t\in [t_{k-1},t_k].
      $$
\end{remark}

Not surprisingly, the terminal value of a Bers-Vekua continuation along any given path is decided by its initial value, as the theorem below shows.
\begin{theorem}\label{thm-uniqueness-Bers-Vekua-continuation}
If two Bers-Vekua continuations $\phi_t, \psi_t$ $(0\leq t\leq 1)$ along a path $\gamma$ in $\mathbb{C}$ share the same initial value, i.e., $\phi_0=\psi_0$, then they are identically equal, i.e., $\phi_t= \psi_t$ $(0\leq t\leq 1)$.
\end{theorem}
\begin{proof}
Consider the following subset of the interval $[0,1]$.
\begin{equation}\label{eq-proof-uniqueness-Bers-Vekua-continuation-1}
  T^*:=\{t^*\in[0,1]|\ \phi_{t^*}=\psi_{t^*}\}.
\end{equation}
Firstly, the given condition $\phi_0=\psi_0$ implies $T^*$ is not empty.
Secondly, Definition \ref{def-Bers-Vekua-continuation} indicates that for every point $t^*\in T^*$, there is an open neighborhood $T\subset[0,1]$ of $t^*$, an open set $O\subset\mathbb{C}$ with $\gamma(T)\subset O$ and a pair $(u,v)\in \mathcal{BV}(O)$ such that
      $$\phi_t=[(u,v)]_{\gamma(t)}=\psi_t\quad\text{for every }t\in T,
      $$
which means $T^*$ is an open subset of $[0,1]$.
It only remains to show $T^*$ is also a closed subset of $[0,1]$.

The final part of this proof is rather complicated. Some basic assumptions and observations are as follows:
Now we assume that $t'$ is an accumulation point of $T^*$. Thus there is a sequence of points $t^*_n\in T^*\setminus\{t'\}$ $(n=1,2,\cdots)$ such that $\lim\limits_{n\to+\infty}t^*_n=t'$.
In addition, according to Remark \ref{rem-Bers-Vekua-continuation}, one can easily see that there exist a partition $0=t_0<t_0<\cdots<t_n=1$ of the interval $[0,1]$ and open sets $O_k\subset \mathbb{O}$ with $\gamma([t_{k-1},t_k])\subset O_k$ and pairs $(u_k,v_k), (u'_k,v'_k)\in\mathcal{BV}(O_k)$ $(k=1,\cdots,n)$ such that
\begin{equation}\label{eq-proof-thm-uniqueness-Bers-Vekua-continuation-2}
  \phi_t=[(u_k,v_k)]_{\gamma(t)},\quad\psi_t=[(u'_k,v'_k)]_{\gamma(t)}\quad\text{for every }t\in [t_{k-1},t_k].
\end{equation}
We notice that there is a (not necessarily unique) sub-interval $[t_{k'-1},t_{k'}]$  which contains infinitely many of $t^*_n$.
Now we focus on such sub-interval; and without loss of generality we may assume it contains all $t^*_n$ $(n=1,2,\cdots)$.

The main steps of the final part are as follows:
Firstly, the assumption that $t^*_n\in[t_{k'-1},t_{k'}]$ converges to $t'$ indicates $t'\in[t_{k'-1},t_{k'}]$.
Thus \eqref{eq-proof-thm-uniqueness-Bers-Vekua-continuation-2}  with $t=t', k=k'$ gives
\begin{equation}\label{eq-proof-uniqueness-Bers-Vekua-continuation-4}
\phi_{t'}=[(u_k,v_k)]_{\gamma(t')},\quad\psi_{t'}=[(u'_k,v'_k)]_{\gamma(t')}.
\end{equation}
Secondly, \eqref{eq-proof-uniqueness-Bers-Vekua-continuation-1} indicates $\phi_{t^*_n}=\psi_{t^*_n}$ since $t^*_n\in T^*$, and \eqref{eq-proof-thm-uniqueness-Bers-Vekua-continuation-2} with $t=t^*_n, k=k'$  gives
$$\phi_{t^*_n}=[(u_{k'},v_{k'})]_{\gamma(t^*_n)},\quad\psi_{t^*_n}=[(u'_{k'},v'_{k'})]_{\gamma(t^*_n)}. $$
We thus have
\begin{equation*}
[(u_{k'},v_{k'})]_{\gamma(t^*_n)}=[(u'_{k'},v'_{k'})]_{\gamma(t^*_n)},
\end{equation*}
which yields $u_{k'}=u'_{k'}$,  $v_{k'}=v'_{k'}$ in the connected component of the open set $O_{k'}$ containing $\gamma(t^*_n)$, due to the identity theorem for real analytic functions (see, e.g., Lemma 5.22 of \cite{Kuchment-2016}).
Furthermore, we notice that $\gamma(t')$ and $\gamma(t^*_n)$ belong to the same connected component of $O_{k'}$ since $\gamma([t_{k'-1},t_k'])\subset O_{k'}$ and $t',t^*_n\in [t_{k'-1},t_{k'}]$.
It follows immediately that $u_{k'}=u'_{k'}$, $v_{k'}=v'_{k'}$ in some open neighborhood of $\gamma(t')$, or equivalently,
$$
[(u_k,v_k)]_{\gamma(t')}=[(u'_k,v'_k)]_{\gamma(t')}.
$$
Substituting the equality above into \eqref{eq-proof-uniqueness-Bers-Vekua-continuation-4}, we obtain $\phi_{t'}=\psi_{t'}$, meaning $t'\in T^*$. In conclusion, $T^*$ is a closed subset of $[0,1]$, which completes the proof.
\end{proof}

\subsection{Uniqueness of stem vectors on CCL-equivalence classes}
In the previous context, we have introduced the concepts of CCL-equivalence and Bers-Vekua continuation.
They serve as our main tools to investigate the uniqueness of stem vectors.
Recall that for an open ball $B$ in the definition domain of any given slice function $f$, we let the symbol $(u_B,v_B)$ represent the corresponding local stem function.
Moreover, combining theorems \ref{thm-real-analytic-of-stem-function} and \ref{thm-equiv-slice-Fueter-and-Vekua}, one can easily see that if $f$ is slice Fueter-regular,
then $(u_B,v_B)$ is real analytic and satisfies \eqref{eq-Vekua}, i.e., $(u_B,v_B)\in\mathcal{BV}(O)$ with $O:=\cup_{I\in\mathbb{S}}B^I$. This observation ensures the well-definedness of the following concept.
\begin{definition}\label{def-I-germ}
Let $f:\Omega\to\mathbb{O}$ be a slice Fueter-regular function. For any $z\in\mathbb{C}$ and $I\in\mathbb{S}$ with $z_I\in\Omega$, the Bers-Vekua germ
$$
[(u_B,v_B)]_z
$$
with $B$ being an arbitrary open ball in $\Omega$ containing $z_I$
is called the $I$-germ of $f$ at the point $z$, denoted by $\phi_f(z,I)$.
\end{definition}
Recall that for any domain $\Omega\subset\mathbb{O}$ and any $I\in\mathbb{S}$,  $\Omega^I$ denotes the set $\{z\in\mathbb{C}|\ z_I:=\tau_I(z)\in\Omega\}$.
The following theorem makes a connection between slice Fueter-regular functions and Bers-Vekua continuations.
\begin{theorem}\label{thm-con-between-SFR-function-&-BV-continuation}
Assume $f:\Omega\to\mathbb{O}$ is a slice Fueter-regular function, and $\gamma_\Theta$ is a circular lifting in $\Omega$, i.e., $\gamma_\Theta([0,1])\subset\Omega$. Then
$$
\phi_t:=\phi_f(\gamma(t),\Theta(t))\quad (0\leq t\leq 1)
$$
is a Bers-Vekua continuation along the path $\gamma$.
\end{theorem}
\begin{proof}
For any $t^*\in[0,1]$, Definition \ref{def-I-germ} gives
$$\phi_{t^*}=\phi_f(\gamma(t^*),\Theta(t^*))=[(u_B,v_B)]_{\gamma(t^*)}$$
where $B$ is an arbitrary open ball in $\Omega$ containing $\gamma_\Theta(t^*)$.
Consider the neighborhood $T:=(\gamma_\Theta)^{-1}B$ of $t^*$ and the open set $O:=\cup_{I\in\mathbb{S}}B^I$ in $\mathbb{C}$.
Firstly, $O$ includes $\gamma(T)$ by definition.
Secondly, as mentioned earlier, the local stem function $(u_B,v_B)$ belongs to $\mathcal{BV}(O)$.
Due to Definition \ref{def-Bers-Vekua-continuation}, it only remains to prove
\begin{equation}\label{eq-thm-con-between-SFR-function-&-BV-continuation-1}
\phi_t=[(u_B,v_B)]_{\gamma(t)} \quad\text{for every }t\in T.
\end{equation}

For every $t\in T$,
Definition \ref{def-I-germ} also gives
\begin{equation}\label{eq-thm-con-between-SFR-function-&-BV-continuation-2}
\phi_{t}=\phi_f(\gamma(t),\Theta(t))=[(u_{B'},v_{B'})]_{\gamma(t)}
\end{equation}
where $B'$ is an arbitrary open ball in $\Omega$ containing $\gamma_\Theta(t)$.
Furthermore, combining Definitions \ref{def-stem-function}, \ref{def-global-stem-function} and \ref{def-local-stem-function}, we obtain
$$
(u_B(z),v_B(z))=(u_{z_I},\rho(z) v_{z_I})
$$
for any $z\in\mathbb{C}$ and $I\in\mathbb{S}$ with $z_I\in B$, where $(u_{z_I}, v_{z_I})$ is the stem vector of $f$ at the point $z_I$, and $\rho$ is a function defined in the following way:  $\rho(z):=+1$ if the imaginary part of $z$ is positive, $\rho(z):=-1$ otherwise.
Similarly, we have
$$
(u_{B'}(z),v_{B'}(z))=(u_{z_I},\rho(z) v_{z_I})
$$
for any $z\in\mathbb{C}$ and $I\in\mathbb{S}$ with $z_I\in B'$.
The preceding two equalities leads to the conclusion that if $z_I\in B\cap B'$, then we have
$$
(u_B(z),v_B(z))=(u_{B'}(z),v_{B'}(z)),
$$
or equivalently,
\begin{equation}\label{eq-thm-con-between-SFR-function-&-BV-continuation-3}
(u_B,v_B)\equiv(u_{B'},v_{B'})\quad\text{on }O',
\end{equation}
where $O':=\cup_{I\in\mathbb{S}}(B\cap B')^I$. Now we claim that the open set $O'$ contains $\gamma(t)$, which in light of \eqref{eq-thm-con-between-SFR-function-&-BV-continuation-3} implies
$$
[(u_B,v_B)]_{\gamma(t)}=[(u_{B'},v_{B'})]_{\gamma(t)}
$$
Substituting the equality above into \eqref{eq-thm-con-between-SFR-function-&-BV-continuation-2}, we finally obtain \eqref{eq-thm-con-between-SFR-function-&-BV-continuation-1}.

The last part of this proof is devoted to verifying what we claimed. For convenience, we may consider $t\in T$ as a constant for now, and take $z=\gamma(t)$ and $I=\Theta(t)$.
Firstly, Definitions \ref{def-circular-lifting} gives $\gamma_\Theta(t)=z_I$, indicating $z_I\in B'$ since $\gamma_\Theta(t)\in B'$.
Moreover, the definition of the set $T$ says $\gamma_\Theta(t)\in B$, meaning $z_I\in B$.
Therefore, $z_I\in B\cap B'$, which leads to the desired conclusion $\gamma(t)\in O'$.
It completes the proof.
\end{proof}

Now we are in a position to characterize the conditional uniqueness of stem vectors of slice Fueter-regular functions in terms of the CCL-equivalence relation.
\begin{theorem}\label{thm-conditional-uniqueness-stem-vectors}
Assume $f:\Omega\to\mathbb{O}$ is a slice Fueter-regular function. Then for any $x, x'\in\Omega$ with $x\overset{\Omega}{\simeq}x'$, the stem vectors of $f$ at $x,x'$ are identical.
\end{theorem}
In other words, the stem vector of a slice Fueter-regular function on every CCL-equivalence class is certainly unique.
\begin{proof}
Based on the construction of local stem function (see Definitions \ref{def-stem-function}, \ref{def-global-stem-function} and \ref{def-local-stem-function}), it suffices to show that there exist
some $z\in\mathbb{C}$ and $I,I'\in\mathbb{S}$ with $z_I=x, z_{I'}=x'$ such that
\begin{equation}\label{eq-proof-thm-conditional-uniqueness-stem-vectors-1}
(u_B(z),v_B(z))=(u_{B'}(z),v_{B'}(z)),
\end{equation}
where $B$ and $B'$ are two arbitrary open balls in $\Omega$ containing $x$ and $x'$ respectively.

Thanks to Definition \ref{def-CCL-equivalence}, without loss of generality, we may replace the condition $x\overset{\Omega}{\simeq}x'$ by the stronger one $x\overset{\Omega}{\sim}x'$, namely $x$ and $x'$ are CCL-connected in $\Omega$. Then there is a pair of coupled circular liftings $\gamma_{\Theta}, \gamma_{\Theta'}$ such that $\gamma_{\Theta}([0,1])$ and $\gamma_{\Theta'}([0,1])$ are included by $\Omega$
and $\gamma_{\Theta}(1)=x,\gamma_{\Theta'}(1)=x'$
due to Definition \ref{def-CCL-connectedness}.
Applying Theorem \ref{thm-con-between-SFR-function-&-BV-continuation} to the circular liftings $\gamma_{\Theta}$ and $\gamma_{\Theta'}$, we see
$$
\phi_{t}:=\phi_f(\gamma(t),\Theta(t))\quad (0\leq t\leq 1)
$$
and
$$
\phi'_{t}:=\phi_f(\gamma(t),\Theta'(t))\quad (0\leq t\leq 1)
$$
are both Bers-Vekua continuations along the path $\gamma$. Now we claim $\phi_{0}=\phi'_{0}$.
It follows immediately from Theorem \ref{thm-uniqueness-Bers-Vekua-continuation} that $\phi_{1}=\phi'_{1}$, which means
\begin{equation}\label{eq-proof-thm-conditional-uniqueness-stem-vectors-2}
[(u_B,v_B)]_{\gamma(1)}=\phi_f(\gamma(1),\Theta(1))=\phi_f(\gamma(1),\Theta'(1))=[(u_{B'},v_{B'})]_{\gamma(1)}
\end{equation}
where $B$ and $B'$ are two arbitrary open balls in $\Omega$ containing $x$ and $x'$ respectively due to  Definition \ref{def-I-germ}.
Taking $z:=\gamma(1)$, $I:=\Theta(1)$ and $I':=\Theta'(1)$,  one can easily see $z_I=x$ and $z_{I'}=x'$.
Finally, \eqref{eq-proof-thm-conditional-uniqueness-stem-vectors-2} yields
$$
(u_B(z),v_B(z))=(u_{B'}(z),v_{B'}(z)),
$$
which is exactly the desired equality \eqref{eq-proof-thm-conditional-uniqueness-stem-vectors-1}.

It only remains to verify our claim. According to the definition of coupled circular liftings (see Definition \ref{def-CCL}), we know that the spherical coordinates $\Theta$ and $\Theta'$ have the same initial value, i.e., $\Theta(0)=\Theta'(0)$.
Hence,
$$\phi_f(\gamma(0),\Theta(0))=\phi_f(\gamma(0),\Theta'(0)),$$
or equivalently, $\phi_{0}=\phi'_{0}$. It completes the proof.
\end{proof}

The theorem above ensures the existence of the global stem functions of slice Fueter-regular functions defined on domains of a special type.
\begin{definition}
A domain $\Omega\subset\mathbb{O}$ is said to be CCL-symmetric if any two points $x, x'\in\Omega$ with $\Re(x)=\Re(x')$ and $\lvert\Im(x)\rvert=\lvert\Im(x')\rvert$ are CCL-equivalent.
\end{definition}
\begin{corollary}\label{cor-complete-on-CCL-symmetric-domain}
Every slice Fueter-regular function defined on a CCL-symmetric domain is a complete slice function.
\end{corollary}
\begin{proof}
Assume that $\Omega$ is a CCL-symmetric domain, and $f:\Omega\to\mathbb{O}$ is a slice Fueter-regular function.
Then for any $x, x'\in\Omega$ with $\Re(x)=\Re(x')$ and $\lvert\Im(x)\rvert=\lvert\Im(x')\rvert$, the stem vectors of $f$ at $x,x'$ are identical due to Theorem \ref{thm-conditional-uniqueness-stem-vectors}.
Hence, the $I_1$-slice stem function are identically equal to the $I_2$-slice stem function on $\Omega^{I_1}\cap\Omega^{I_2}$ for any $I_1,I_2\in\mathbb{S}$ by Definition \ref{def-stem-function}, which completes the proof.
\end{proof}
\begin{remark}\label{Re-star-shaped-domain}
It is not difficult to see that if a domain $\Omega$ is star-shaped with respect to some point in the real axis, then it is CCL-symmetric. Thus there exist a large number of CCL-symmetric domains.
\end{remark}
Now we would like to give an explicit example of a non-CCL-symmetric domain.
First, we construct some open balls
$$B_\theta:=\left\{x\in\mathbb{O}\left|\lvert x-2\phi(\theta)-\exp(\theta\phi(\theta))\rvert<\frac{1}{4}\right.\right\}.
$$
Here $\phi(\theta)=I\cos\frac{\theta}{2}+J\sin\frac{\theta}{2}$ $(-\pi\leq\theta\leq\pi)$, and $I,J$ are two imaginary units with $I\perp J$.
The next thing to do is to construct the domain
$$
\Omega:=\cup_{\theta\in[-\pi,\pi]}B_\theta.
$$
Recall that in the end of Sect. \ref{subsection-slice-Fueter-function} we have presented an example of a non-complete slice Fueter-regular function whose definition domain is exactly this $\Omega$ above.
Finally, we arrive at the conclusion that $\Omega$ is a non-CCL-symmetric domain due to Corollary \ref{cor-complete-on-CCL-symmetric-domain}. This example shows that the concept ``CCL-symmetric domain" is non-trivial.

\section{CCL Riemann domains}
We notice that the CCL equivalence relation on every domain $\Omega\subset\mathbb{O}$ naturally induces a corresponding Riemann domain which can be considered as a unified definition domain of the stem functions of all slice Fueter-regular functions on $\Omega$. Some characteristics of this Riemann domain will be discussed in this part of our paper.
\subsection{A variant of CCL equivalence}
For any given domain $\Omega\subset\mathbb{O}$, we denote the disjoint union of its complex slices by $\hat{\Omega}$.
Technically, it can be expressed in the following form.
$$
\hat{\Omega}:=\bigcup_{I\in\mathbb{S}}\Omega^I\times\{I\}=\{(z,I)\in\mathbb{C}\times\mathbb{S}|\ z_I\in\Omega \}.
$$
By convention, a subset $O$ of $\hat{\Omega}$ is said to be open if and only if there are open subsets $O^I$ of $\Omega^I$ $(I\in\mathbb{S})$ such that
$$
O=\bigcup_{I\in\mathbb{S}}O^I\times\{I\}.
$$
And such topology of $\hat{\Omega}$ is compatible with the distance defined as
$$
d((z_1,I_1),(z_2,I_2)):=\left\{\begin{array}{ll}
                                 \frac{\lvert z_1-z_2\rvert}{\lvert z_1-z_2\rvert +1}, & \text{if }I_1=I_2,\\
                                 & \\
                                 1, & \text{if }I_1\neq I_2.
                               \end{array}\right.
$$

It is not difficult to see that the CCL equivalence relation on $\Omega$ can be converted into an equivalence relation on $\hat{\Omega}$.
\begin{definition}\label{def-induced-CCL-equivalence}
Let $\Omega$ be a domain in $\mathbb{O}$.
Two elements $(z,I),(z',I')\in\hat{\Omega}$ is said to be CCL-equivalent in $\hat{\Omega}$, denoted as $(z,I) \overset{\hat{\Omega}}{\simeq} (z',I')$, precisely if $z=z'$ and the octonions $z_I$ and $z'_{I'}$ are CCL-equivalent in $\Omega$.
\end{definition}
Briefly speaking, $(z,I) \overset{\hat{\Omega}}{\simeq} (z',I')$ iff $z=z'$ and $z_I \overset{\Omega}{\simeq} z'_{I'}$.

For convenience, we denote the open disc in $\mathbb{C}$ with center $z$ and radius $\delta$ by $U(z;\delta)$.
The local infectivity of CCL equivalence relation is shown in the theorem below.
\begin{theorem}\label{thm-local-infection-pattern-1}
Assume $(z,I_1) \overset{\hat{\Omega}}{\simeq} (z,I_2)$. Then there exists a $\delta>0$ such that  $(z',I_1) \overset{\hat{\Omega}}{\simeq} (z',I_2)$ holds for any $z'\in U(z;\delta)$.
\end{theorem}
\begin{proof}
In the case that $z=a\in\mathbb{R}$, the proof is short.
In such case, observing a given ($8$-dimensional) open ball $B$ in $\Omega$ centered at $a$,
one can easily see that $B$ contains all $z'_{I_k}$ $(k=1,2)$ with $z'\in U(a;\delta)$, where $\delta$ is the radius of $B$.
Note that $B$ is star-shaped with respect to the point $a$ in the real axis.
Hence, as stated in Remark \ref{Re-star-shaped-domain}, $B$ is certainly CCL-symmetric, meaning $z'_{I_k}$ $(k=1,2)$ are CCL-equivalent, i.e.,
$$(z',I_1) \overset{\hat{\Omega}}{\simeq} (z',I_2).$$

Now we turn to the case that $z\notin\mathbb{R}$.
By Definition \ref{def-CCL-equivalence}, there exists a sequence of octonions $z_{I_1}=x_0,x_1,\cdots,x_n=z_{I_2}$ such that
$$x_ {k-1}\overset{\Omega}{\sim}x_k\quad\text{for } k=1,\cdots,n. $$
More explicitly, there are $n$ pairs of coupled circular liftings ${(\gamma_k)}_{\Theta_{k}}$, ${(\gamma_k)}_{\Theta'_{k}}$ such that
\begin{equation}\label{eq-proof-thm-local-infection-1}
 {(\gamma_k)}_{\Theta_{k}}(1)=x_{k-1},\ {(\gamma_k)}_{\Theta'_{k}}(1)=x_k \quad\text{for } k=1,\cdots, n,
\end{equation}
which is due to Definition \ref{def-CCL-connectedness}.
Without loss of generality, we may assume that $\gamma_k(1)=z$, because we can always replace the base $\gamma_k$ by its complex conjugate and attach the minus sign to the spherical coordinates $\Theta_{k}, \Theta'_{k}$ when some $\gamma_k(1)=\overline{z}$.
Denote $I'_0=I_1$, $I'_n=I_2$ and  $I'_k=\Theta'_{k}(1)$ for the rest of $k$'s.
Then by \eqref{eq-proof-thm-local-infection-1} we have
$$
x_k=z_{I'_k}\quad\text{for }  k=0,1,\cdots, n.
$$
Moreover, since $\Omega$ is open, there exists a positive number $\delta$ such that $\Omega$ includes all the open balls centered at $x_k$ $(k=0,1,\cdots,n)$ with radius $\delta$, indicating
$$U(z;\delta)\times \{I'_k\}\subset\hat{\Omega}. $$
Now we claim that $z'_{I'_{k-1}}$, $z'_{I'_k}$ $(k=1,\cdots, n)$ are CCL-connected, i.e.,
\begin{equation}\label{eq-proof-thm-local-infection-2}
z'_{I'_{k-1}}\overset{\Omega}{\sim}z'_{I'_k},
\end{equation}
for any $z'\in U(z;\delta)$.
It follows immediately that
$z'_{I_1}\overset{\Omega}{\simeq}z'_{I_2}$, or equivalently,
$$
(z',I_1) \overset{\hat{\Omega}}{\simeq} (z',I_2).
$$

To complete the proof in the second case, it only remains to verify \eqref{eq-proof-thm-local-infection-2}.
Note that the conditions $\gamma_k(1)=z\notin\mathbb{R}$ and ${(\gamma_k)}_{\Theta_{k}}(1)=x_{k-1}=z_{I'_{k-1}}$ ensure $\Theta_{k}(1)=I'_{k-1}$.
Define a path in $\mathbb{C}$ as
$$
\mu_k(t):=\left\{\begin{array}{ll}
              \gamma_k(2t), & 0\leq t\leq \frac{1}{2},\\
               (2-2t)z+(2t-1)z', & \frac{1}{2}<t\leq 1;\\
            \end{array}\right.
$$
and two pathes in $\mathbb{S}$ as
$$
\Psi_k(t):=\left\{\begin{array}{ll}
              \Theta_k(2t), & 0\leq t\leq \frac{1}{2},\\
               I'_{k-1}, & \frac{1}{2}<t\leq 1;\\
            \end{array}\right.\quad
\Psi'_k(t):=\left\{\begin{array}{ll}
              \Theta'_k(2t), & 0\leq t\leq \frac{1}{2},\\
               I'_{k}, & \frac{1}{2}<t\leq 1. \\
            \end{array}\right.
$$
It is easy to see that ${(\mu_k)}_{\Psi_k}$, ${(\mu_k)}_{\Psi'_k}$ are coupled circular liftings  in the domain $\Omega$. Hence their terminal points are CCL-connected, meaning  \eqref{eq-proof-thm-local-infection-2} is valid.
\end{proof}

\subsection{The quotient space under CCL equivalence}
For any domain $\Omega\subset\mathbb{O}$,
we let $[(z,I)]_{\Omega}$ denote the equivalence class of an element $(z,I)\in\hat{\Omega}$ under the equivalence relation $\overset{\hat{\Omega}}{\simeq}$, and $\mathcal{D}_{\Omega}$ denote the quotient space of $\hat{\Omega}$ by $\overset{\hat{\Omega}}{\simeq}$.
Additionally, we denote the canonical projection from $\hat{\Omega}$ to $\mathcal{D}_{\Omega}$ by $\pi$.
By convention, the quotient space $\mathcal{D}_{\Omega}$ is endowed with the finest topology which makes $\pi$ a continuous mapping.
In brief, a set $O\subset\mathcal{D}_{\Omega}$ is open iff $\pi^{-1}O$ is open in $\hat{\Omega}$.

The local infectivity of CCL equivalence (see Theorem \ref{thm-local-infection-pattern-1}) leads to the next theorem.
\begin{theorem}\label{thm-openness-of-pi}
The canonical projection $\pi:\hat{\Omega}\to\mathcal{D}_{\Omega}$ is open.
\end{theorem}
\begin{proof}
Due to the construction of topologies of $\hat{\Omega}$ and $\mathcal{D}_{\Omega}$, it suffices to show that $\pi^{-1}\left(\pi(A\times\{I\})\right)$ is open for any nonempty open subset $A$ of every complex slice $\Omega^I$ $(I\in\mathbb{S})$.
Firstly, by definition, we see that
$$(z,I')\in\pi^{-1}(\pi(A\times\{I\})), $$
if and only if
$$z\in A \quad \text{and}\quad (z,I') \overset{\hat{\Omega}}{\simeq} (z,I). $$
Furthermore, for such $(z,I')$ and $(z,I)$ above, Theorem \ref{thm-local-infection-pattern-1} says there is a $\delta>0$ such that
$$(z',I') \overset{\hat{\Omega}}{\simeq} (z',I)$$
holds for any $z'\in U(z;\delta)$.
In addition, there is a $\delta'>0$ such that $U(z;\delta')\subset A$ since $A$ is open. It follows immediately that
$$z'\in A \quad \text{and}\quad (z',I') \overset{\hat{\Omega}}{\simeq} (z',I)$$
holds for any $z'\in U(z;\delta'')$,
where $\delta'':=\min\{\delta,\delta'\}$.
In other words,
$$(z',I')\in\pi^{-1}(\pi(A\times\{I\}))$$
for any $z'\in U(z;\delta'')$.
We thus know  $\pi^{-1}\left(\pi(A\times\{I\})\right)$ includes the open neighborhood $U(z;\delta'')\times\{I'\}$ of $(z,I')$, which means every element of $\pi^{-1}\left(\pi(A\times\{I\})\right)$ is an interior point. The proof is now completed.
\end{proof}
\subsection{The complex structure on $\mathcal{D}_\Omega$ induced by $\pi$} 
We notice that the canonical projection $\pi$ naturally induces a family of mappings from the complex slices $\Omega^I$ $(I\in\mathbb{S})$ to the quotient space $\mathcal{D}_{\Omega}$ as follows. 
\begin{equation*}
\pi^I(z):=\pi(z,I)=[(z,I)]_{\Omega}. 
\end{equation*}
Some basic properties of $\pi^I:\Omega^I\to\mathcal{D}_{\Omega}$ are showed in the lemmas below.
\begin{lem}\label{lemma-pi-I-embedding}
$\pi^I:\Omega^I\to\mathcal{D}_{\Omega}$ $(I\in\mathbb{S})$ are topological embeddings.
\end{lem}
\begin{proof}
Firstly, it is easy to see that if $[(z,I)]_{\Omega}=[(z',I)]_{\Omega}$, then $z=z'$, indicating that
every $\pi^I$ is injective.
Secondly, considering the mapping
$$\sigma^I: \Omega^I\to\hat{\Omega}, \quad z\mapsto (z,I), $$
we have
$$\pi^I=\pi\circ\sigma^I. $$
Apparently, the second factor $\sigma^I$ is a topological embedding due to the construction of the topology of $\hat{\Omega}$, which implies that $\sigma^I$ is continuous and open.
In addition, the first factor $\pi$ is continuous due to the construction of the topology of $\mathcal{D}_{\Omega}$, and open as shown by Theorem \ref{thm-openness-of-pi}.
In conclusion, every $\pi^I$ is injective, continuous and open, which completes the proof.
\end{proof}

\begin{lem}\label{lemma-pi-I-holomorphic-compatible}
$\pi^I:\Omega^I\to\mathcal{D}_{\Omega}$ $(I\in\mathbb{S})$ are holomorphically compatible, i.e., the mapping
$$(\pi^{I'})^{-1}\circ\pi^I: (\pi^I)^{-1}O^{I,I'}\to(\pi^{I'})^{-1}O^{I,I'}
$$
is holomorphic for any $I, I'\in \mathbb{S}$ with $O^{I,I'}\neq\emptyset$,
where $O^{I,I'}:=\pi^{I'}(\Omega^{I'})\cap\pi^{I}(\Omega^{I})$
\end{lem}
\begin{proof}
The conclusion of this lemma is quite trivial, since the previous lemma implies that all sets involved in this lemma are open and a simple calculation yields $(\pi^{I'})^{-1}\circ\pi^I(z)=z$ for all  $z\in(\pi^I)^{-1}O^{I,I'}$.
Indeed, $(\pi^I)^{-1}O^{I,I'}$ and $(\pi^{I'})^{-1}O^{I,I'}$ are the same, and $(\pi^{I'})^{-1}\circ\pi^I$ is the identity mapping from such open set to itself, which obviously indicates that it is holomorphic.
\end{proof}

As a direct result of Lemmas \ref{lemma-pi-I-embedding} and  \ref{lemma-pi-I-holomorphic-compatible}, we see that the embeddings $\pi^I$ $(I\in\mathbb{S})$ induce a complex structure on $\mathcal{D}_\Omega$.
\begin{theorem}\label{thm-existence-complex-altas}
The family of mappings $(\pi^I)^{-1}:\pi^I(\Omega^I)\to\Omega^I$ $(I\in\mathbb{S})$
is a complex atlas on $\mathcal{D}_\Omega$.
\end{theorem}
\begin{proof}
Firstly, Lemma \ref{lemma-pi-I-embedding} implies $\pi^I(\Omega^I)$ is open, and the relation between $\pi$ and $\pi^I$ $(I\in\mathbb{S})$ yields
$$\cup_{I\in\mathbb{S}}\ \pi^I(\Omega^I)=\pi(\hat{G})=\mathcal{D}_\Omega.
$$
Thus we see that $\left\{\pi^I(\Omega^I)\right\}_{I\in \mathbb{S}}$ is an open covering of $\mathcal{D}_\Omega$.
Secondly, Lemma \ref{lemma-pi-I-embedding} indicates that the mappings $(\pi^I)^{-1}:\pi^I(\Omega^I)\to\Omega^I$ $(I\in\mathbb{S})$ are homeomorphisms, and Lemma \ref{lemma-pi-I-holomorphic-compatible} indicates that they are holomorphically compatible.
 Therefore, $(\pi^I)^{-1}:\pi^I(\Omega^I)\to\Omega^I$ $(I\in\mathbb{S})$
is a complex atlas on $\mathcal{D}_\Omega$.
\end{proof}
For convenience, we let $\Pi_\Omega$ denote the complex structure generated by the complex atlas in Theorem \ref{thm-existence-complex-altas}.
\subsection{The natural local homeomorphism from $\mathcal{D}_\Omega$ to $\mathbb{C}$}
There is a natural mapping from $\mathcal{D}_\Omega$ to $\mathbb{C}$ defined as
$$
P: [(z,I)]_\Omega \mapsto z.
$$
This mapping is obviously well-defined, since
$$[(z,I)]_\Omega=[(z',I')]_\Omega \quad\Longleftrightarrow\quad (z,I)\overset{\hat{\Omega}}{\simeq}(z',I') \quad\Longrightarrow\quad z=z'. $$
A direct calculation yields that $P\circ\pi^I$ is the identity mapping on $\Omega^I$ $(I\in\mathbb{S})$. In light of Theorem \ref{thm-existence-complex-altas}, we thus see that $P$ is not only a holomorphic function (w.r.t. the complex structure $\Pi_\Omega$), but also a local homeomorphism. Then it follows immediately that the quotient space $\mathcal{D}_\Omega$ equipped with the local homeomorphism $P$ is a Riemann domain over $\mathbb{C}$ (see, e.g., Sec. G of \cite{Gunning-1965}). These conclusions are summarized as follows.
\begin{theorem}\label{thm-CCL-Riemann-domain}
$P$ is holomorphic w.r.t. the complex structure $\Pi_\Omega$, and $(\mathcal{D}_\Omega,P)$ is a Riemann domain over $\mathbb{C}$.
\end{theorem}
We name $(\mathcal{D}_\Omega,P)$ as the CCL Riemann domain associated with $\Omega$.
\subsection{A connection between circular liftings in $\Omega$ and liftings in $\mathcal{D}_\Omega$}
Recall the classical definition of path lifting (see, e.g., Pages 22-23 of \cite{Forster-1981}):

Let $\gamma:[0,1]\to X$, $\Gamma:[0,1]\to Y$ be pathes in the topological spaces $X$ and $Y$ respectively, and $p:Y\to X$ be a local homeomorphism. Then $\Gamma$ is said to be a lifting of $\gamma$ w.r.t. $p$ if $\gamma=p\circ\Gamma$.

One can easily see that a circular lifting is not strictly a ``real" lifting.
But the theorem below shows  that every circular lifting in any given domain $\Omega\subset\mathbb{O}$  induces a lifting in the CCL Riemann domain associated with $\Omega$ in a natural way.
\begin{theorem}\label{thm-from-CL-to-L}
Assume $\gamma_\Theta$ is a circular lifting in a domain $\Omega\subset\mathbb{O}$, i.e., the range of $\gamma_\Theta$ is included by $\Omega$. Then the mapping $\Gamma$ defined as
$$
\Gamma: t\mapsto [(\gamma(t),\Theta(t))]_\Omega \quad (0\leq t\leq 1)
$$
is a path in $\mathcal{D}_\Omega$ satisfying $\gamma=P\circ\Gamma$, i.e., $\Gamma$ is a lifting of $\gamma$ w.r.t. $P$.
\end{theorem}
Note that here $\gamma$ is the base of $\gamma_\Theta$ and $P$ is the natural local homeomorphism from $\mathcal{D}_\Omega$ to $\mathbb{C}$.
\begin{proof}
Firstly, we will show that $\Gamma$ is well-defined. For any $t\in [0,1]$, it is easy to see that $(\gamma(t),\Theta(t))$ is an element in $\hat{\Omega}$, since the range of $\gamma_\Theta$ is included by $\Omega$ as assumed. Hence $\Gamma$ is a well-defined mapping from $[0,1]$ to $\mathcal{D}_\Omega$.
Now we are going to prove that $\Gamma$ is continuous. Briefly speaking, the continuity of $\gamma$ and $\Theta$ ensures that of $\Gamma$. Indeed, for any $t^*\in[0,1]$, there is a $\delta>0$ such that
$$
z_I\in \Omega
$$
for any $z\in\gamma\left((t^*-\delta,t^*+\delta)\cap[0,1]\right)$ and $I\in\Theta\left((t^*-\delta,t^*+\delta)\cap[0,1]\right)$,
which leads to $$z_{I_1}\overset{\Omega}{\simeq}z_{I_2},$$
or equivalently,
$$
[(z,I_1)]_\Omega=[(z,I_2)]_\Omega,
$$
for any $z\in\gamma\left((t^*-\delta,t^*+\delta)\cap[0,1]\right)$ and $I_1, I_2\in\Theta\left((t^*-\delta,t^*+\delta)\cap[0,1]\right)$, since $z_{I_1}$ and $z_{I_2}$ are connected by a circular lifting in $\Omega$ with a constant base.
Taking
$$z=\gamma(t), I_1=\Theta(t)\text{ and } I_2=\Theta(t^*)$$
 with $t\in(t^*-\delta,t^*+\delta)\cap[0,1]$, we thus see
$$[(\gamma(t),\Theta(t))]_\Omega=[(\gamma(t),\Theta(t^*))]_\Omega, $$
meaning
$$\Gamma(t)=[(\gamma(t),\Theta(t^*))]_\Omega$$
holds for any $t\in(t^*-\delta,t^*+\delta)\cap[0,1]$.
Therefore,
$$(\pi^{I})^{-1}\Gamma\equiv\gamma\quad\text{on }(t^*-\delta,t^*+\delta)\cap[0,1],$$
when $I=\Theta(t^*)$.
In light of Theorem \ref{thm-existence-complex-altas}, it follows that $\Gamma$ is continuous on an open neighborhood of any point $t^*$ in the interval $[0,1]$.
Moreover, a direct calculation yields $\gamma=P\circ\Gamma$. The proof is now completed.
\end{proof}

The above theorem shows a strong connection between (octonionic) circular liftings and liftings in CCL Riemann domains. The connection also reveal a topological characteristic of CCL Riemann domains as shown below.
\begin{corollary}
The CCL Riemann domain associated with every domain in $\mathbb{O}$ has at most two connected components.
\end{corollary}
\begin{proof}
The existence of complex structure (see Theorem \ref{thm-existence-complex-altas}) implies that every CCL Riemann domain is locally path-connected.
We thus know the connected components coincide with the path-connected components in every CCL Riemann domain.
So it suffices to prove that at least two of any three elements in every CCL Riemann domain are path-connected.

Assume $\Omega$ is a domain in $\mathbb{O}$, and $[(z_k,I_k)]_\Omega$ $(k=1,2,3)$ are three arbitrary points in $\mathcal{D}_{\Omega}$, i.e., the CCL Riemann domain associated with $\Omega$.
Then by definition we see all the three octonions $(z_k)_{I_k}$ belong to $\Omega$.
It thus follows from Theorem \ref{thm-connectedness-via-CL} that there exist three circular liftings in $\Omega$ connecting each two of $(z_k)_{I_k}$ $(k=1,2,3)$ respectively.
For convenience, we denote the base and the spherical coordinate of the circular lifting connecting $(z_i)_{I_i}$ and $(z_j)_{I_j}$ $(1\leq i<j\leq 3)$ by $\gamma_{i,j}$ and $\Theta_{i,j}$.
We would like to mention that one can always replace $\gamma_{i,j}$ and $\Theta_{i,j}$ by $\overline{\gamma_{i,j}}
$ and $-\Theta_{i,j}$ without changing the corresponding circular lifting. Hence there are essentially two possible cases as follows:

Case $1$: At least one of $\gamma_{i,j}$ $(1\leq i<j\leq 3)$ satisfies $\gamma_{i,j}(0)=z_i$ and $\gamma_{i,j}(1)=z_j$.
Then Theorem \ref{thm-from-CL-to-L} indicates that for such $i, j$ there exists a path in $\mathcal{D}_{\Omega}$ connecting
$[(z_i,I_i)]_\Omega$ and $[(z_j,I_j)]_\Omega$.

Case $2$: Every $\gamma_{i,j}$ $(1\leq i<j\leq 3)$ satisfies $\gamma_{i,j}(0)=z_i$ and $\gamma_{i,j}(1)=\overline{z_j}$.
Then invoking Theorem \ref{thm-from-CL-to-L} again, we see there exist a path in $\mathcal{D}_{\Omega}$ connecting
$[(z_1,I_1)]_\Omega$ and $[(\overline{z_3},-I_3)]_\Omega$, and another path in  $\mathcal{D}_{\Omega}$ connecting $[(z_2,I_2)]_\Omega$ and $[(\overline{z_3},-I_3)]_\Omega$, implying that
$[(z_1,I_1)]_\Omega$ and $[(z_2,I_2)]_\Omega$ are path-connected.

In conclusion, no matter in which case, at least two of $[(z_k,I_k)]_\Omega$ $(k=1,2,3)$ are path-connected in $\mathcal{D}_{\Omega}$. It completes the proof.
\end{proof}
\begin{remark}
We would like to illustrate the theorem above by two simple examples as follows.
If we look at an open ball in $\mathbb{O}$ with a center far away from the real axis and a sufficiently small radius, then we can see that the associated  CCL Riemann domain is essentially homeomorphic to the disjoint union of two open discs in $\mathbb{C}$.
In such example, there are two connected components.
By contrast, the CCL Riemann domain associated with an open ball in $\mathbb{O}$ with a center on the real axis is connected, thereby having only one connected component.
\end{remark}

\subsection{Stem functions on $\mathcal{D}_\Omega$}
Theorem \ref{thm-conditional-uniqueness-stem-vectors} implies that the CCL Riemann domain $\mathcal{D}_\Omega$ can serve as a unified definition domain of stem functions of all slice Fueter-regular functions in the domain $\Omega$. Recall that any complex slice $\Omega^I$ can be separated into two parts:
$$\Omega^I_{+}:=\Omega^I\cap\{\alpha+\beta i|\beta\geq 0\}\quad\text{and}\quad\Omega^I_{-}:=\Omega^I\cap\{\alpha+\beta i|\beta < 0\}. $$
We mention here to avoid misunderstanding that $\alpha$ and $\beta$ are real numbers.
\begin{definition}
Let $\Omega$ be a domain in $\mathbb{O}$. For a slice Fueter-regular function $f$ in the domain $\Omega$, the stem function of $f$ on $\mathcal{D}_\Omega$ is defined by
$$
(\hat{u},\hat{v}): [(z,I)]_\Omega\mapsto (u_{z_I}, \pm v_{z_I})  \quad\text{for } z\in\Omega^I_\pm,
$$
where $(u_{z_I}, v_{z_I})$ is the stem vector of $f$ at the point $z_I$.
\end{definition}
Theorem \ref{thm-conditional-uniqueness-stem-vectors} indicates $$(u_{z_I}, v_{z_I})=(u_{{z}_{I'}}, v_{{z}_{I'}})$$
for any $(z,I), (z,I')\in\hat{\Omega}$ with $[(z,I)]_\Omega=[(z,I')]_\Omega$. Hence $(\hat{u},\hat{v})$ is well-defined.
Furthermore, it is not difficult to see the following equality:
$$
f(z_I)=\hat{u}\left([(z,I)]_\Omega\right)+I \hat{v}\left([(z,I)]_\Omega\right)\quad\text{for }(z,I)\in\hat{\Omega}.
$$

\begin{remark}
If $\Omega$ is CCL-symmetric, then the mapping $P:\mathcal{D}_\Omega\to\mathbb{C}$ is in fact a homeomorphism from $\mathcal{D}_\Omega$ to $\cup_{I\in\mathbb{S}}\Omega^I$ $(\subset\mathbb{C})$. Moreover, Corollary \ref{cor-complete-on-CCL-symmetric-domain} ensures the existence of the global stem function $(u_\Omega,v_{\Omega})$ of $f$.
In such case, the stem function $(\hat{u},\hat{v})$ of $f$ on $\mathcal{D}_\Omega$ is essentially identical with the global stem function $(u_\Omega,v_{\Omega})$ in the following sense:
$$
\hat{u}=u_\Omega\circ P\quad\text{and}\quad\hat{v}=v_{\Omega}\circ P.
$$
\end{remark}

\bibliographystyle{plainnat}
\bibliography{mybibfile}
\end{document}